\documentclass[12pt, a4paper, reqno]{amsart}
\usepackage{amsmath}
\usepackage{amsfonts}
\usepackage{amssymb}
\usepackage{color}
\usepackage{comment,graphicx}
\usepackage[T1]{fontenc}
\usepackage{url}
\usepackage{ae}

\def\phi{\varphi}
\def\rho{\varrho}
\def\epsilon{\varepsilon}
\numberwithin{equation}{section}
\theoremstyle{plain}
\newtheorem{theorem}[equation]{Theorem}
\newtheorem{lemma}[equation]{Lemma}

\theoremstyle{definition}
\newtheorem{definition}[equation]{Definition}

\theoremstyle{remark}
\newtheorem{remark}[equation]{Remark}

\renewcommand{\leq}{\leqslant}
\renewcommand{\geq}{\geqslant}

\input{tcilatex}

\begin{document}
\title[Herz-type Sobolev spaces on domains]{Herz-type Sobolev spaces on
domains}
\author{Douadi Drihem}
\maketitle
\date{\today }

\begin{abstract}
We introduce Herz-type Sobolev spaces on domains, which unify and generalize
the classical \ Sobolev spaces. We will give a proof of the Sobolev-type
embedding for these function spaces. All these results generalize the
classical results on Sobolev spaces. Some remarks on
Caffarelli--Kohn--Nirenberg inequality are given.


\end{abstract}

\section{Introduction}

Function spaces have been widely used in various areas of analysis such as
harmonic analysis and partial differential equations. Some example of these
spaces can be mentioned such as Sobolev spaces. The interest in these spaces
comes not only from theoretical reasons but also from their applications in
mathematical analysis. We refer to the monographs \cite{AdamsHedberg96}, 
\cite{AdamsFournier03}, \cite{Burenkov} and \cite{Ma85} for further details,
historical remarks and references on Sobolev spaces.

It is well known that Herz spaces play an important role in harmonic
analysis. After they have been introduced in \cite{Herz68}, the theory of
these spaces had a remarkable development in part due to its usefulness in
applications. For instance, they appear in the characterization of
multipliers on Hardy spaces \cite{BS85}, in the summability of Fourier
transforms \cite{Fi-We08}, in regularity theory for elliptic equations in
divergence form \cite{Rag09}. Also \cite{Tu11}, studied the Cauchy problem
for Navier-Stokes equations on Herz spaces and weak Herz spaces. Recently,
Herz spaces appear in the study of semilinear parabolic equations \cite%
{Dr-Herz-Heat} and summability of Fourier transforms on mixed-norm Lebesgue
spaces \cite{HWDY}. For important and latest results on Herz spaces, we
refer the reader to the papers \cite{RS}, \cite{ZYZ} and to the monograph 
\cite{LYH}.

Based on Sobolev and Herz spaces we present a class of function spaces,
called Herz-type Sobolev spaces, which generalize the classical Sobolev
spaces. These type of function spaces, but over ${\mathbb{R}^{n}}$, are
introduced by Lu and Yang \cite{LuYang97} were gave some applications to
partial differential equations.

In this paper our spaces defined over a domain.\ More precisely the domain
is often assumed to satisfy a cone condition.

The paper is organized as follows. First we give some preliminaries where we
fix some notation and recall some basics facts on Herz spaces, where the
approximation by smooth functions is given. In particular, we will prove the
Herz type version of Caffarelli--Kohn--Nirenberg-type inequalities.

In Section 3, we present basics facts on Herz-type Sobolev spaces in analogy
to the classical Sobolev spaces and we prove a Sobolev embedding theorem for
these spaces. In particular we prove that 
\begin{equation}
\dot{K}_{p,m}^{\alpha _{2},r}(\Omega )\hookrightarrow \dot{K}_{q}^{\alpha
_{1},r}(\Omega )  \label{main-est}
\end{equation}%
with some appropriate assumptions on the parameters. The surprise here is
that the embedding \eqref{main-est} is true if $1<q<p<\infty ,\alpha _{2}+%
\frac{n}{p}\geq \alpha _{1}+\frac{n}{q}>0$ and 
\begin{equation*}
\max \Big(\frac{n}{p},\frac{n}{p}+\alpha _{2},\frac{n}{p}-\frac{n}{q}+\alpha
_{2}-\alpha _{1}\Big)<m<n.
\end{equation*}

The proof based on a local estimate and on the boundedness of maximal
function and Riesz potential operator on Herz spaces. Other properties of
these function spaces such interpolation inequalities, extension and compact
embeddings are postponed to the future work.

\section{Herz spaces}

As usual, $\mathbb{R}^{n}$ denotes the $n$-dimensional real Euclidean space, 
$\mathbb{N}$ the collection of all natural numbers and $\mathbb{N}_{0}=%
\mathbb{N}\cup \{0\}$. The letter $\mathbb{Z}$ stands for the set of all
integer numbers. For any $u>0$, $k\in \mathbb{Z}$ we set $R(u)=\{x\in 
\mathbb{R}^{n}:\frac{u}{2}\leq \left\vert x\right\vert <u\}$ and $%
R_{k}=R(2^{k})$. For $x\in \mathbb{R}^{n}$ and $r>0$ we denote by $B(x,r)$
the open ball in $\mathbb{R}^{n}$ with center $x$ and radius $r$. Let $\chi
_{k}$, for $k\in \mathbb{Z}$, denote the characteristic function of the set $%
R_{k}$. If $1\leq p<\infty $ and $\frac{1}{p}+\frac{1}{p^{\prime }}=1$, then 
$p^{\prime }$ is called the conjugate exponent of $p$.

\noindent We denote by $\left\vert \Omega \right\vert $ the $n$-dimensional
Lebesgue measure of $\Omega \subseteq \mathbb{R}^{n}$. For any measurable
subset $\Omega \subseteq \mathbb{R}^{n}$ the Lebesgue space $L^{p}(\Omega )$%
, $0<p\leq \infty $ consists of all measurable functions for which 
\begin{equation*}
\big\Vert f\big\Vert_{L^{p}(\Omega )}=\Big(\int_{\Omega }\left\vert
f(x)\right\vert ^{p}dx\Big)^{1/p}<\infty ,\text{\quad }0<p<\infty
\end{equation*}%
and 
\begin{equation*}
\big\Vert f\big\Vert_{L^{\infty }(\Omega )}=\underset{x\in \Omega }{\text{%
ess-sup}}\left\vert f(x)\right\vert <\infty .
\end{equation*}%
If $\Omega =\mathbb{R}^{n}$ then we put $\big\Vert f\big\Vert_{L^{p}(\mathbb{%
R}^{n})}=\big\Vert f\big\Vert_{p}$.

Let $\Omega \subseteq \mathbb{R}^{n}$ be open. For any nonnegative integer $%
m $ let $C^{m}(\Omega )$ be the vector space consisting of all functions $f$%
, which, together with all their partial derivatives $D^{\beta }f$ of orders 
$|\beta |\leq m$, are continuous on $\Omega $. We put $C^{0}(\Omega
)=C(\Omega )$ and $C^{\infty }(\Omega )=\cap _{m\geq 0}C^{m}(\Omega )$. We
denote by $C_{c}(\Omega )$\ the set of all functions in $C(\Omega )$ which
have compact support in $\Omega $.

In this section we present some fundamental properties of Herz spaces. We
start by recalling the definition and some of the properties of the
homogenous Herz spaces.

\begin{definition}
\label{def-inh-Herz}Let $\alpha \in \mathbb{R}$ and $1\leq p,q\leq \infty $.
The homogeneous Herz space $\dot{K}_{p}^{\alpha ,q}(\mathbb{R}^{n})$ is
defined as the set of all $f\in L_{\mathrm{loc}}^{p}\left( {\mathbb{R}^{n}}%
\setminus \{0\}\right) $ such that 
\begin{equation*}
\big\Vert f\big\Vert_{\dot{K}_{p}^{\alpha ,q}(\mathbb{R}^{n})}=\Big(%
\sum\limits_{k=-\infty }^{\infty }2^{k\alpha q}\big\|f\,\chi _{k}\big\|%
_{p}^{q}\Big)^{1/q}<\infty
\end{equation*}%
with the usual modifications when $p=\infty $ and/or $q=\infty $.
\end{definition}

The spaces $\dot{K}_{p}^{\alpha ,q}(\mathbb{R}^{n})$ are Banach spaces. If $%
\alpha =0$\ and $1\leq p=q\leq \infty $ then $\dot{K}_{p}^{0,p}(\mathbb{R}%
^{n})$ coincides with the Lebesgue spaces $L^{p}(\mathbb{R}^{n})$. If $1\leq
q_{1}\leq q_{2}\leq \infty $, then we may derive the embedding $\dot{K}%
_{p}^{\alpha ,q_{1}}(\mathbb{R}^{n})\hookrightarrow \dot{K}_{p}^{\alpha
,q_{2}}(\mathbb{R}^{n})$. In addition%
\begin{equation*}
\dot{K}_{p}^{\alpha ,p}(\mathbb{R}^{n})=L^{p}(\mathbb{R}^{n},|\cdot
|^{\alpha p}),\quad \text{(Lebesgue space equipped with power weight),}
\end{equation*}%
where%
\begin{equation*}
\big\|f\big\|_{L^{p}(\mathbb{R}^{n},|\cdot |^{\alpha p})}=\Big(\int_{\mathbb{%
R}^{n}}\left\vert f(x)\right\vert ^{p}|x|^{\alpha p}dx\Big)^{1/p}.
\end{equation*}%
If $\Omega \subset \mathbb{R}^{n}$ is open and $f:\Omega \rightarrow \mathbb{%
R}$ a measurable function, then we write $f\in \dot{K}_{p}^{\alpha
,q}(\Omega )$ if $f\chi _{\Omega }\in \dot{K}_{p}^{\alpha ,q}(\mathbb{R}%
^{n}) $ and we put $\big\|f\big\|_{\dot{K}_{p}^{\alpha ,q}(\Omega )}=\big\|%
f\chi _{\Omega }\big\|_{\dot{K}_{p}^{\alpha ,q}(\mathbb{R}^{n})}$.

Various important results have been proved in the space $\dot{K}_{p}^{\alpha
,q}(\mathbb{R}^{n})$\ under some assumptions on $\alpha ,p$ and $q$. The
conditions $-\frac{n}{p}<\alpha <n(1-\frac{1}{p}),1<p<\infty $ and $1\leq
q\leq \infty $ is crucial in the study of the boundedness of classical
operators in $\dot{K}_{p}^{\alpha ,q}(\mathbb{R}^{n})$ spaces. This fact was
first realized by Li and Yang \cite{LiYang96} with the proof of the
boundedness of the maximal function. As usual, we put%
\begin{equation*}
\mathcal{M(}f)(x)=\sup_{Q}\frac{1}{|Q|}\int_{Q}\left\vert f(y)\right\vert
dy,\quad f\in L_{\mathrm{loc}}^{1}(\mathbb{R}^{n}),
\end{equation*}%
where the supremum\ is taken over all cubes with sides parallel to the axis
and $x\in Q$. Also we set 
\begin{equation*}
\mathcal{M}_{t}(f)=\left( \mathcal{M(}\left\vert f\right\vert ^{t})\right) ^{%
\frac{1}{t}},\quad 0<t<\infty .
\end{equation*}

\begin{lemma}
\label{Maximal-Inq}Let $1<p<\infty $ and $1\leq q\leq \infty $. If $f$ is a
locally integrable functions on $\mathbb{R}^{n}$ and $-\frac{n}{p}<\alpha
<n(1-\frac{1}{p})$, then%
\begin{equation*}
\big\|\mathcal{M}(f)\big\|_{\dot{K}_{p}^{\alpha ,q}(\mathbb{R}^{n})}\leq c%
\big\|f\big\|_{\dot{K}_{p}^{\alpha ,q}(\mathbb{R}^{n})}.
\end{equation*}
\end{lemma}

A detailed discussion of the properties of these spaces my be found in the
recent monograph \cite{LYH08}, the papers \cite{LuYang1.95}, \cite%
{LuYang2.95}, \cite{Hernandez1998}, and references therein.

The next lemma is a Hardy-type inequality which is basically a consequence
of Young's inequality in the sequence Lebesgue space $\ell ^{q}$.

\begin{lemma}
\label{lq-inequality}\textit{Let }$0<a<1\ $\textit{and }$0<q\leq \infty $%
\textit{. Let }$\left\{ \varepsilon _{k}\right\} _{k\in \mathbb{Z}}$\textit{%
\ be a sequences of positive real numbers and denote} $\delta
_{k}=\sum_{j=k}^{\infty }a^{j-k}\varepsilon _{j}$, $k\in \mathbb{Z}$.\textit{%
\ }Then there exists a constant $c>0\ $\textit{depending only on }$a$\textit{%
\ and }$q$ such that%
\begin{equation*}
\big\|\left\{ \delta _{k}\right\} _{k\in \mathbb{Z}}\big\|_{\ell ^{q}}\leq c%
\big\|\left\{ \varepsilon _{k}\right\} _{k\in \mathbb{Z}}\big\|_{\ell ^{q}}.
\end{equation*}
\end{lemma}

Let $V_{\alpha ,p,q}$ be the set of $(\alpha ,p,q)\in \mathbb{R}\times
\lbrack 1,\infty ]^{2}$\ such that:

\textbullet\ $\alpha <n-\frac{n}{p}$, $1\leq p\leq \infty $ and $1\leq q\leq
\infty ,$

\textbullet\ $\alpha =n-\frac{n}{p}$, $1\leq p\leq \infty \ $and$\ q=1,$

The next lemma gives a necessary and sufficient condition on the parameters $%
\alpha $, $p$ and $q$, in order to make sure that 
\begin{equation*}
\langle T_{f},\varphi \rangle =\int_{\Omega }f(x)\varphi (x)dx,\quad \varphi
\in \mathcal{D}(\Omega ),f\in \dot{K}_{p}^{\alpha ,q}(\Omega )
\end{equation*}%
generates a regular distribution $T_{f}\in \mathcal{D}^{\prime }(\Omega )$.

\begin{lemma}
\label{L1loc}\ Let $\Omega \subset \mathbb{R}^{n}$ be open, $0\in \Omega $
and $1\leq p,q\leq \infty $. Then 
\begin{equation*}
\dot{K}_{p}^{\alpha ,q}(\Omega )\hookrightarrow L_{\mathrm{loc}}^{1}(\Omega
),
\end{equation*}
\end{lemma}

if and only if $(\alpha ,p,q)\in V_{\alpha ,p,q}.$

\begin{proof}
We divide the proof into two steps.

\textbf{Step 1.} Assume that $(\alpha ,p,q)\in V_{\alpha ,p,q}$, $f\in \dot{K%
}_{p}^{\alpha ,q}(\Omega )$ and $B(0,2^{N})\subset \Omega ,N\in \mathbb{Z}$.
By similarity we only consider the first case. H\"{o}lder's inequality gives%
\begin{align*}
\big\|f\big\|_{L^{1}(B(0,2^{N}))} &=\sum_{i=-\infty }^{N}\big\|f\chi
_{R_{i}\cap \Omega }\big\|_{1} \\
&\lesssim \sum_{i=-\infty }^{N}2^{i(n-\frac{n}{p})}\big\|f\chi _{R_{i}\cap
\Omega }\big\|_{p} \\
&=c2^{N(n-\frac{n}{p}-\alpha )}\sum_{i=-\infty }^{N}2^{(i-N)(n-\frac{n}{p}%
-\alpha )}2^{i\alpha }\big\|f\chi _{R_{i}\cap \Omega }\big\|_{p} \\
&\lesssim \big\|f\big\|_{\dot{K}_{p}^{\alpha ,q}(\Omega )}.
\end{align*}

\textbf{Step 2.} Assume that $(\alpha ,p,q)\notin V_{\alpha ,p,q}$. We
distinguish two cases.

\textbf{Case 1.} $\alpha >n-\frac{n}{p}$. Let $r>0$ be such that $%
B(0,r)\subset \Omega $ and set $f(x)=|x|^{-n}\chi _{0<|\cdot |<r}(x)$. We
obtain $f\in \dot{K}_{p}^{\alpha ,q}(\Omega )$ for any $1\leq p,q\leq \infty 
$\ whereas $f\notin L_{\mathrm{loc}}^{1}(\Omega )$. Indeed, we find%
\begin{align*}
\big\|f\big\|_{\dot{K}_{p}^{\alpha ,q}(\Omega )}^{q} &=\sum_{k\in \mathbb{Z}%
:2^{k}<2r}2^{k\alpha q}\big\|f\chi _{R_{k}\cap \Omega }\big\|_{p}^{q} \\
&\lesssim \sum_{k\in \mathbb{Z}:2^{k}<2r}2^{k(\alpha -n)q}\big\|\chi
_{R_{k}}\chi _{0<|\cdot |<r}\big\|_{p}^{q} \\
&\lesssim \sum_{k\in \mathbb{Z}:2^{k}<2r}2^{k(\alpha -n+\frac{n}{p})q} \\
&<\infty ,
\end{align*}%
with the usual modification if $p=\infty $ and/or $q=\infty $. Obviously, $%
f\notin L_{\mathrm{loc}}^{1}(\Omega )$.

\textbf{Case 2.} $\alpha =n-\frac{n}{p}$, $1\leq p\leq \infty $ and $1<q\leq
\infty $. By similarity we can assume that $B(0,\frac{1}{2})\subset \Omega $%
. We consider the function $f$ defined by%
\begin{equation*}
f(x)=|x|^{-n}(|\log |x||)^{-1}\chi _{0<|\cdot |<\frac{1}{2}}(x).
\end{equation*}%
An easy computation yields that%
\begin{equation*}
\big\|f\big\|_{\dot{K}_{p}^{\alpha ,q}(\Omega )}^{q}\lesssim
\sum_{k=1}^{\infty }k^{-q}<\infty ,
\end{equation*}%
which gives that $f\in \dot{K}_{p}^{n-\frac{n}{p},q}(\Omega )$, with the
usual modifications when $q=\infty $. It is easily seen that $f$ does not
belong to $L_{\mathrm{loc}}^{1}(\Omega )$.
\end{proof}

\begin{remark}
We easily see that in general if $0\notin \Omega $ then the set $V_{\alpha
,p,q}$ is not optimal. From this lemma it thus makes sense to talk about
weak derivatives of functions in $\dot{K}_{p}^{\alpha ,q}(\Omega )$, in
addition the assumption $(\alpha ,p,q)\in V_{\alpha ,p,q}$ is optimal.
\end{remark}

\begin{theorem}
\label{dense}Let $\Omega \subset \mathbb{R}^{n}$\ be open and $0\in \Omega $%
, $1<p<\infty ,1\leq q<\infty $ and $\alpha >-\frac{n}{p}$. Then $%
C_{c}(\Omega )$ is dense in $\dot{K}_{p}^{\alpha ,q}(\Omega )$.
\end{theorem}

\begin{proof}
First observe that $C_{c}(\Omega )\subset \dot{K}_{p}^{\alpha ,q}(\Omega )$
if and only if $\alpha >-\frac{n}{p}$. Indeed, let $\varphi \in C_{c}(\Omega
)$ be such that $\varphi (x)=1,x\in B(0,2^{N})\subset \Omega ,N\in \mathbb{Z}
$. We have%
\begin{align*}
\big\|\varphi \big\|_{\dot{K}_{p}^{\alpha ,q}(\Omega )}^{q}&
=\sum_{k=-\infty }^{\infty }2^{k\alpha q}\big\|\varphi \chi _{\Omega \cap
R_{k}}\big\|_{p}^{q} \\
& \geq \sum_{k=-\infty }^{N}2^{k\alpha q}\big\|\chi _{B(0,2^{N})\cap R_{k}}%
\big\|_{p}^{q} \\
& =\sum_{k=-\infty }^{N}2^{k\alpha q}\big\|\chi _{R_{k}}\big\|_{p}^{q} \\
& =c\sum_{k=-\infty }^{N}2^{k(\alpha +\frac{n}{p})q}
\end{align*}%
and this series is divergent if $\alpha \leq -\frac{n}{p}$. It is clear that 
$C_{c}(\Omega )\subset \dot{K}_{p}^{\alpha ,q}(\Omega )$ whenever $\alpha >-%
\frac{n}{p}$. Let $\dot{K}_{p,c}^{\alpha ,q}(\Omega )$ be the set of all $%
g\in \dot{K}_{p}^{\alpha ,q}(\Omega )$ such that $g=0$ outside a compact. As
in \cite[Proposition 3.1]{Zuily}\ we obtain that $\dot{K}_{p,c}^{\alpha
,q}(\Omega )$ is dense in $\dot{K}_{p}^{\alpha ,q}(\Omega )$. Therefore we
prove the density of $C_{c}(\Omega )$ in $\dot{K}_{p,c}^{\alpha ,q}(\Omega )$%
. Let $f\in \dot{K}_{p,c}^{\alpha ,q}(\Omega )$ with $f(x)=0$ if $x\notin
A\subset \Omega $ compact. As in \cite[Theorem 2.19]{AdamsFournier03}, the
proof can be restricted to the case $f$ is real-valued and nonnegative.
Since $f$ is measurable, there exists a monotonically increasing sequence $%
\{u_{i}\}_{i\in \mathbb{N}_{0}}$ of nonnegative simple functions converging
pointwise to $f$ on $\Omega $ and%
\begin{equation*}
0\leq u_{i}\leq f,\quad i\in \mathbb{N}_{0}.
\end{equation*}%
Since 
\begin{equation*}
0\leq f-u_{i}\leq f,\quad i\in \mathbb{N}_{0},
\end{equation*}%
by dominated convergence theorem $\{u_{i}\}_{i\in \mathbb{N}_{0}}$ converge
to $f$ in $\dot{K}_{p}^{\alpha ,q}(f)$. Therefore we find an $u\in
\{u_{i}\}_{i\in \mathbb{N}_{0}}$ such that%
\begin{equation*}
\big\|f-u\big\|_{\dot{K}_{p}^{\alpha ,q}(\Omega )}<\frac{\varepsilon }{2}.
\end{equation*}%
Since $0\leq u\leq f$, $\mathrm{supp}u\subset A$. Let $\theta >0$ be such
that $\max (0,\frac{-\alpha p}{n})<\theta <1$. Assume that $A\subset
V\subset \bar{V}\subset \Omega $ with $\bar{V}$ compact. We set%
\begin{equation*}
E=\sum_{k\in \mathbb{Z}:R_{k}\cap \bar{V}\neq \emptyset }2^{k\alpha
q}|R_{k}|^{\frac{\theta q}{p}}.
\end{equation*}%
By Lusin's theorem we can find that $\varphi \in C_{c}(\Omega )$ such that%
\begin{equation*}
|\varphi (x)|\leq \big\|u\big\|_{\infty }
\end{equation*}%
for any $x\in \Omega $, supp$\varphi \subset \bar{V}$ and%
\begin{equation*}
|H|\leq \Big(\frac{\varepsilon }{4\big\|u\big\|_{\infty }E^{\frac{1}{q}}}%
\Big)^{\frac{p}{1-\theta }},
\end{equation*}%
where $H=\{x\in \Omega :\varphi (x)\neq u(x)\}$. We set $B=\{x\in \bar{V}%
:\varphi (x)\neq u(x)\}$. Observe that $H=B$. We have%
\begin{align*}
\big\|u-\varphi \big\|_{\dot{K}_{p}^{\alpha ,q}(\Omega )}^{q}&
=\sum_{k=-\infty }^{\infty }2^{k\alpha q}\big\|(u-\varphi )\chi _{R_{k}\cap
\Omega }\big\|_{p}^{q} \\
& =\sum_{k=-\infty }^{\infty }2^{k\alpha q}\big\|(u-\varphi )\chi
_{R_{k}\cap H}\big\|_{p}^{q} \\
& =\sum_{k=-\infty }^{\infty }2^{k\alpha q}\big\|(u-\varphi )\chi
_{R_{k}\cap B}\big\|_{p}^{q}.
\end{align*}%
Therefore%
\begin{equation*}
\big\|u-\varphi \big\|_{\dot{K}_{p}^{\alpha ,q}(\Omega )}^{q}=\sum_{k\in 
\mathbb{Z}:R_{k}\cap B\neq \emptyset }2^{k\alpha q}\big\|(u-\varphi )\chi
_{R_{k}\cap B}\big\|_{p}^{q}.
\end{equation*}%
Let $k\in \mathbb{Z}$ be such that $R_{k}\cap B\neq \emptyset $. Then%
\begin{align*}
\big\|(u-\varphi )\chi _{R_{k}\cap B}\big\|_{p}& \leq 2\big\|u\big\|_{\infty
}\big\|\chi _{R_{k}\cap B}\big\|_{p} \\
& =2\big\|u\big\|_{\infty }\big\|\chi _{R_{k}\cap B}\big\|_{p}^{1-\theta }%
\big\|\chi _{R_{k}\cap B}\big\|_{p}^{\theta } \\
& \leq 2\big\|u\big\|_{\infty }\big\|\chi _{R_{k}}\big\|_{p}^{\theta }\big\|%
\chi _{B}\big\|_{p}^{1-\theta } \\
& \leq 2|R_{k}|^{\frac{\theta }{p}}\big\|u\big\|_{\infty }|B|^{\frac{%
1-\theta }{p}}.
\end{align*}%
Consequently, 
\begin{align*}
\big\|u-\varphi \big\|_{\dot{K}_{p}^{\alpha ,q}(\Omega )}^{q}& \leq 2^{q}%
\text{ }E\big\|u\big\|_{\infty }^{q}|B|^{\frac{1-\theta }{p}q} \\
& <(\frac{\varepsilon }{2})^{q},\quad \varepsilon >0
\end{align*}%
and that ends the proof.
\end{proof}

\begin{theorem}
\label{K-separable}Let $\Omega $ be open,\ $1\leq p<\infty ,\ 1\leq q<\infty 
$ and $\alpha >-\frac{n}{p}$. Then $\dot{K}_{p}^{\alpha ,q}(\Omega )$ is
separable.
\end{theorem}

\begin{proof}
As in \cite[Lemma 2.17]{Lieb-Ross} it suffices to prove the theorem for $%
\Omega =\mathbb{R}^{n}$. For $j\in \mathbb{N}$ and $m=(m_{1},...,m_{n})\in 
\mathbb{Z}^{n}$ let%
\begin{equation*}
Q_{j,m}=\big\{x\in \mathbb{R}^{n}:2^{-j}m_{i}\leq x<2^{-j}(m_{i}+1),i=1,...,n%
\big\}
\end{equation*}%
be the dyadic cube. Put%
\begin{equation*}
F_{j}=\big\{f:f=\sum_{m\in \mathbb{Z}^{n}}a_{j,m}\chi _{Q_{j,m}},a_{j,m}\in 
\mathbb{Q}\big\},\quad j\in \mathbb{N},
\end{equation*}%
where $a_{j,m}=0$ if $|m|\geq N,N\in \mathbb{N}$. We have $F=\cup _{j\in 
\mathbb{N}}F_{j}$, is a countable set. Let $f\in \dot{K}_{p}^{\alpha ,q}(%
\mathbb{R}^{n})$ and $\varepsilon >0$. From Theorem \ref{dense} there exists 
$\varphi \in C_{c}(\mathbb{R}^{n})$ such that%
\begin{equation*}
\big\|f-\varphi \big\|_{\dot{K}_{p}^{\alpha ,q}(\mathbb{R}^{n})}\leq \frac{%
\varepsilon }{2}.
\end{equation*}%
Assume that $\mathrm{supp}\varphi \subset Q_{-J,z}$, $J\in \mathbb{N},z\in 
\mathbb{Z}^{n}$ with $J$ large enough. Let $j\in \mathbb{N},m\in \mathbb{Z}%
^{n}$ and%
\begin{equation*}
\varphi _{j,m}(x)=\left\{ 
\begin{array}{ccc}
2^{-jn}\int_{Q_{j,m}}\varphi (y)dy, & \text{if} & x\in Q_{j,m}\subseteq
Q_{-J,z}, \\ 
0, & \text{if} & x\in Q_{j,m}\nsubseteq Q_{-J,z}\text{ or }x\notin Q_{j,m}.%
\end{array}%
\right.
\end{equation*}%
Observe that%
\begin{equation*}
\big\|\varphi -\varphi _{j,m}\big\|_{\dot{K}_{p}^{\alpha ,q}(\mathbb{R}%
^{n})}^{q}=\sum_{k=-\infty }^{\infty }2^{k\alpha q}\big\|(\varphi -\varphi
_{j,m})\chi _{k}\big\|_{p}^{q}.
\end{equation*}%
But%
\begin{equation*}
\big\|(\varphi -\varphi _{j,m})\chi _{k}\big\|_{p}^{p}=\int_{Q_{-J,z}}|%
\varphi (x)-\varphi _{j,m}(x)|^{p}\chi _{k}(x)dx
\end{equation*}%
for any $j\in \mathbb{N}$ and any $m\in \mathbb{Z}^{n}$. Since $\varphi $ is
uniformly continuous on $Q_{-J,z}$, for each $\varepsilon ^{\prime }>0$
there is a $\delta >0$ such that%
\begin{equation*}
|\varphi (x)-\varphi (y)|<(\varepsilon ^{\prime })^{p}
\end{equation*}%
whenever $|x-y|<\delta $. Let $x\in Q_{-J,z}$. We can find a dyadic cube $%
Q_{j,m_{1}}$ such that $x\in Q_{j,m_{1}}\subseteq Q_{-J,z}$ for any $j\in 
\mathbb{N}$. We have%
\begin{equation*}
|\varphi (x)-\varphi _{j,m_{1}}(x)|\leq 2^{-jn}\int_{Q_{j,m_{1}}}|\varphi
(x)-\varphi (y)|dy,\quad x\in Q_{j,m_{1}}\subseteq Q_{-J,z}
\end{equation*}%
for any $j\in \mathbb{N}$. Taking $j$ large enough be such that $|x-y|\leq 
\sqrt{n}2^{-j}\leq \delta $, $x,y\in Q_{j,m_{1}}$. Let $j_{1}$ one of them.
Therefore%
\begin{equation*}
|\varphi (x)-\varphi _{j_{1},m_{1}}(x)|<(\varepsilon ^{\prime })^{p},\quad
x\in Q_{j_{1},m_{1}}\subseteq Q_{-J,z}.
\end{equation*}%
Hence%
\begin{align*}
\big\|\varphi -\varphi _{j_{1},m_{1}}\big\|_{\dot{K}_{p}^{\alpha ,q}(\mathbb{%
R}^{n})}^{q} &=\sum_{2^{k}\lesssim (1+|z|)2^{J}}2^{k\alpha q}\big\|(\varphi
-\varphi _{j_{1},m_{1}})\chi _{R_{k}\cap Q_{-J,z}}\big\|_{p}^{q} \\
&\leq \sum_{2^{k}\lesssim (1+|z|)2^{J}}2^{k(\alpha +\frac{n}{p})q}\sup_{x\in
Q_{j_{1},m_{1}}}|\varphi (x)-\varphi _{j_{1},m_{1}}(x)|^{\frac{q}{p}} \\
&\leq c(\varepsilon ^{\prime })^{q}((1+|z|)2^{J})^{(\alpha +\frac{n}{p})q},
\end{align*}%
with the help of the fact that $\alpha >-\frac{n}{p}$. Since $\varphi
_{j_{1},m_{1}}(x)\in \mathbb{R}$ we can find $\tilde{\varphi}%
_{j_{1},m_{1}}(x)\in \mathbb{Q}$\ be such that 
\begin{equation*}
|\varphi _{j_{1},m_{1}}(x)-\tilde{\varphi}_{j_{1},m_{1}}(x)|<\varepsilon
^{\prime },\quad x\in Q_{j_{1},m_{1}}\subseteq Q_{-J,z}.
\end{equation*}%
Now%
\begin{align*}
\big\|\varphi -\tilde{\varphi}_{j_{1},m_{1}}\big\|_{\dot{K}_{p}^{\alpha ,q}(%
\mathbb{R}^{n})} &\leq \big\|\varphi -\varphi _{j_{1},m_{1}}\big\|_{\dot{K}%
_{p}^{\alpha ,q}(\mathbb{R}^{n})}+\big\|\tilde{\varphi}_{j_{1},m_{1}}-%
\varphi _{j_{1},m_{1}}\big\|_{\dot{K}_{p}^{\alpha ,q}(\mathbb{R}^{n})} \\
&\leq C\varepsilon ^{\prime }((1+|z|)2^{J})^{(\alpha +\frac{n}{p})q}.
\end{align*}%
We choose $\varepsilon ^{\prime }$ be such that $C\varepsilon ^{\prime
}((1+|z|)2^{J})^{\alpha +\frac{n}{p}}<\frac{\varepsilon }{2}$, which yields
that%
\begin{equation*}
\big\|f-\tilde{\varphi}_{j_{1},m_{1}}\big\|_{\dot{K}_{p}^{\alpha ,q}(\mathbb{%
R}^{n})}\leq \varepsilon .
\end{equation*}%
This completes the proof.
\end{proof}

Let $J\in \mathcal{D}(\mathbb{R}^{n})$ be a real-valued function with 
\begin{equation*}
J(x)\geq 0,\quad \text{if}\quad x\in \mathbb{R}^{n},\quad J(x)=0\quad \text{%
if}\quad x\in \overline{B(0,1)}\quad \text{and}\quad \int_{\mathbb{R}%
^{n}}J(x)dx=1.
\end{equation*}%
We put $J_{\varepsilon }(x)=\varepsilon ^{-n}J(\frac{x}{\varepsilon }),$ $%
x\in \mathbb{R}^{n}$.

\begin{theorem}
\label{approximation1}Let $\Omega \subset \mathbb{R}^{n}$ be open, $0\in
\Omega ,1\leq p<\infty ,\ 1\leq q<\infty $ and $-\frac{n}{p}<\alpha <n-\frac{%
n}{p}$. Let $f\in \dot{K}_{p}^{\alpha ,q}(\Omega )\ $be a function defined
on $\mathbb{R}^{n}$ and vanishes identically outside $\Omega $. Then%
\begin{equation}
\lim_{\varepsilon \rightarrow 0_{+}}\big\|J_{\varepsilon }\ast f-f\big\|_{%
\dot{K}_{p}^{\alpha ,q}(\Omega )}=0.  \label{whatweneed}
\end{equation}
\end{theorem}

\begin{proof}
We will do the proof into two steps.

\textbf{Step 1.} We will prove that 
\begin{equation}
\lim_{\varepsilon \rightarrow 0_{+}}\big\|J_{\varepsilon }\ast \varphi
-\varphi \big\|_{\dot{K}_{p}^{\alpha ,q}(\Omega )}=0  \label{limit1}
\end{equation}%
for any $\varphi \in C_{c}(\Omega )$ and $\alpha \in \mathbb{R}$.

\textbf{Substep 1.1.} $\alpha >-\frac{n}{p}$. Assume that $\mathrm{supp}%
\varphi \subset B(0,2^{N})\subset \Omega $, $N\in \mathbb{N}$. Using the
fact that $|x-y|>2^{N}\geq \varepsilon $ for any $x\in \mathbb{R}%
^{n}\backslash B(0,2^{N+1})$ and any $y\in B(0,2^{N})$ we obtain 
\begin{equation*}
J_{\varepsilon }\ast \varphi (x)=0,\quad x\in \mathbb{R}^{n}\backslash
B(0,2^{N+1}),\quad \varepsilon \leq 2^{N},
\end{equation*}%
which yields 
\begin{equation*}
\big\|(J_{\varepsilon }\ast \varphi -\varphi )\chi _{k}\big\|_{L^{p}(\Omega
)}^{p}=\int_{\Omega \cap R_{k}\cap B(0,2^{N+1})}\big|J_{\varepsilon }\ast
\varphi (x)-\varphi (x)\big|^{p}dx,\quad k\in \mathbb{Z}.
\end{equation*}%
Observe that%
\begin{equation*}
J_{\varepsilon }\ast \varphi (x)-\varphi (x)=\int_{B(0,1)}J(z)(\varphi
(x-\varepsilon z)-\varphi (x))dz.
\end{equation*}%
Therefore%
\begin{equation*}
|J_{\varepsilon }\ast \varphi (x)-\varphi (x)|\leq \sup_{z\in
B(0,1)}|\varphi (x-\varepsilon z)-\varphi (x)|
\end{equation*}%
which tend to zero as $\varepsilon \rightarrow 0$. Hence%
\begin{align*}
\big\|J_{\varepsilon }\ast \varphi -\varphi \big\|_{\dot{K}_{p}^{\alpha
,q}(\Omega )}^{q} &\leq \sum_{k=-\infty }^{N+2}2^{k\alpha q}\big\|%
(J_{\varepsilon }\ast \varphi -\varphi )\chi _{k}\big\|_{p}^{q} \\
&\leq \sup_{|x|\leq 2^{N+2}}\sup_{z\in B(0,1)}|\varphi (x-\varepsilon
z)-\varphi (x)|^{q}\sum_{k=-\infty }^{N+2}2^{k(\alpha +\frac{n}{p})q} \\
&\lesssim \sup_{|x|\leq 2^{N+2}}\sup_{z\in B(0,1)}|\varphi (x-\varepsilon
z)-\varphi (x)|^{q}.
\end{align*}%
Letting $\varepsilon $ tend to zero, we obtain \eqref{limit1} for any $%
\varphi \in C_{c}(\Omega )$ and $\alpha >-\frac{n}{p}$.

\textbf{Substep 1.2.} $\alpha <-\frac{n}{p}$. By duality%
\begin{equation*}
\big\|J_{\varepsilon }\ast \varphi -\varphi \big\|_{\dot{K}_{p}^{\alpha
,q}(\Omega )}=\sup \Big|\int_{\Omega }(J_{\varepsilon }\ast \varphi
(x)-\varphi (x))g(x)dx\Big|,
\end{equation*}%
where the supremum is taken over all continuous functions of compact support 
$g$ such that $\big\|g\big\|_{\dot{K}_{p^{\prime }}^{-\alpha ,q\prime
}(\Omega )}=1$. It is easily seen that%
\begin{equation*}
\int_{\Omega }(J_{\varepsilon }\ast \varphi (x)-\varphi
(x))g(x)dx=\int_{\Omega }(\tilde{J}_{\varepsilon }\ast g(x)-g(x))\varphi
(x)dx,
\end{equation*}%
where $\tilde{J}_{\varepsilon }(x)=J_{\varepsilon }(-x),x\in \mathbb{R}^{n}$%
. We have%
\begin{equation*}
\Big|\int_{\Omega }(\tilde{J}_{\varepsilon }\ast g(x)-g(x))\varphi (x)dx\Big|%
\leq \big\|\tilde{J}_{\varepsilon }\ast g-g\big\|_{\dot{K}_{p^{\prime
}}^{-\alpha ,q^{\prime }}(\Omega )}\big\|\varphi \big\|_{\dot{K}_{p}^{\alpha
,q}(\Omega )}.
\end{equation*}%
Observe that $-\alpha >-\frac{n}{p^{\prime }}$. Using Substep 1.1, we see
that%
\begin{equation*}
\big\|\tilde{J}_{\varepsilon }\ast g-g\big\|_{\dot{K}_{p^{\prime }}^{-\alpha
,q^{\prime }}(\Omega )}\leq \frac{\eta }{\big\|\varphi \big\|_{\dot{K}%
_{p}^{\alpha ,q}(\Omega )}}
\end{equation*}%
for any $\eta >0$ and any $\varepsilon $ small enough. Hence%
\begin{equation*}
\big\|J_{\varepsilon }\ast \varphi -\varphi \big\|_{\dot{K}_{p}^{\alpha
,q}(\Omega )}\leq \eta
\end{equation*}%
for any $\eta >0$ and any $\varepsilon $ small enough.

\textbf{Substep 1.3.} $\alpha =-\frac{n}{p}$. Let $\alpha _{0}>-\frac{n}{p}$
and $\alpha _{1}<-\frac{n}{p}$ be such that $\alpha =\theta \alpha
_{0}+(1-\theta )\alpha _{1},0<\theta <1$. H\"{o}lder's inequality yields%
\begin{align*}
\big\|J_{\varepsilon }\ast \varphi -\varphi \big\|_{\dot{K}_{p}^{\alpha
,q}(\Omega )} &\leq \big\|J_{\varepsilon }\ast \varphi -\varphi \big\|_{\dot{%
K}_{p}^{\alpha _{0},q}(\Omega )}^{\theta }\big\|J_{\varepsilon }\ast \varphi
-\varphi \big\|_{\dot{K}_{p}^{\alpha _{1},q}(\Omega )}^{1-\theta } \\
&\leq \eta
\end{align*}%
for any $\eta >0$ and any $\varepsilon $ small enough.

\textbf{Step 2. }We prove \eqref{whatweneed}. By Theorem \ref{dense} we can
find $\varphi \in C_{c}(\Omega )$ such that 
\begin{equation*}
\big\|f-\varphi \big\|_{\dot{K}_{p}^{\alpha ,q}(\Omega )}\leq \frac{\eta }{3}
\end{equation*}%
for any $\eta >0$ small enough. So, for any $\eta _{1}>0$ small enough 
\begin{equation*}
\big\|J_{\varepsilon }\ast f-J_{\varepsilon }\ast \varphi \big\|_{\dot{K}%
_{p}^{\alpha ,q}(\Omega )}\leq c\big\|\mathcal{M}(f-\varphi )\big\|_{\dot{K}%
_{p}^{\alpha ,q}(\mathbb{R}^{n})}\leq c\eta _{1}
\end{equation*}%
by Lemma \ref{Maximal-Inq}, because of $-\frac{n}{p}<\alpha <n-\frac{n}{p}$.
We choose $\eta _{1}$ be such that $c\eta _{1}<\frac{\eta }{3}$. From Step 1,%
\begin{equation*}
\big\|J_{\varepsilon }\ast \varphi -\varphi \big\|_{\dot{K}_{p}^{\alpha
,q}(\Omega )}\leq \frac{\eta }{3}
\end{equation*}%
by choosing $\varepsilon $ sufficiently small, which prove \eqref{whatweneed}
but with $p>1$. Let $s>1$. H\"{o}lder's inequality and the fact that $%
-n<\alpha <0$ yield%
\begin{equation*}
\big\|J_{\varepsilon }\ast f-f\big\|_{\dot{K}_{1}^{\alpha ,q}(\Omega )}\leq %
\big\|J_{\varepsilon }\ast f-f\big\|_{\dot{K}_{s}^{\alpha +n-\frac{n}{s}%
,q}(\Omega )},
\end{equation*}%
which tends to zero as $\varepsilon \rightarrow 0_{+}$.

This completes the proof.
\end{proof}

Let $1\leq q<\infty $. The Caffarelli--Kohn--Nirenberg inequality says that 
\begin{equation*}
\Big(\int_{\mathbb{R}^{n}}|x|^{\gamma p}|f(x)|^{p}dx\Big)^{\frac{1}{p}}\leq c%
\Big(\int_{\mathbb{R}^{n}}|x|^{\alpha q}|\nabla f(x)|^{q}dx\Big)^{\frac{1}{q}%
}
\end{equation*}%
for any $f\in \mathcal{D}(\mathbb{R}^{n})$, where 
\begin{equation}
\alpha >-\frac{n}{q},\quad \gamma >-\frac{n}{p},\quad \alpha -1\leq \gamma
\leq \alpha ,\quad \frac{n}{p}-\frac{n}{q}=\alpha -\gamma -1\leq 0,
\label{CKN-condition}
\end{equation}%
see \cite{CKN84}. This inequality plays an important role in theory of
function spaces and PDE's. Our aim is to extend this result to Herz spaces.

\begin{theorem}
\label{Embeddings1}Let $1\leq q\leq \frac{n}{n-1},0<r\leq \infty $\ and%
\begin{equation*}
\alpha _{2}+n-1=\alpha _{1}+\frac{n}{q}>0.
\end{equation*}%
Then%
\begin{equation}
\big\|f\big\|_{\dot{K}_{q}^{\alpha _{1},r}(\mathbb{R}^{n})}\lesssim \big\|f%
\big\|_{\dot{W}_{1,1}^{\alpha _{2},r}(\mathbb{R}^{n})},\quad f\in \mathcal{D}%
(\mathbb{R}^{n}),  \label{KCN}
\end{equation}%
holds, where%
\begin{equation}
\big\|f\big\|_{\dot{W}_{1,1}^{\alpha _{2},r}(\mathbb{R}^{n})}=\Big(%
\sum\limits_{k=-\infty }^{\infty }2^{k\alpha _{2}r}\big\|(\nabla f)\chi _{k}%
\big\|_{1}^{r}\Big)^{\frac{1}{r}}.  \label{homogenous Herz-Sobolev}
\end{equation}
\end{theorem}

\begin{proof}
Since $\alpha _{2}+n>0$, \eqref{homogenous Herz-Sobolev} is well defined and
finite for any $f\in \mathcal{D}(\mathbb{R}^{n})$.\ Let 
\begin{equation*}
I_{1,k}=\big[-2^{k},2^{k}\big],\quad I_{2,k}=\big[\frac{1}{\sqrt{n}}%
2^{k-1},2^{k}\big],\quad k\in \mathbb{Z},
\end{equation*}%
and%
\begin{equation*}
I_{3,k}=\big[-2^{k},-\frac{1}{\sqrt{n}}2^{k-1}\big],\quad I_{4,k}=\Big(-%
\frac{1}{\sqrt{n}}2^{k-1},\frac{1}{\sqrt{n}}2^{k-1}\Big),\quad k\in \mathbb{Z%
}.
\end{equation*}%
We set%
\begin{equation*}
J_{k}=\cup _{i=1}^{n-1}V_{i,k}\cup V_{k},\quad k\in \mathbb{Z},
\end{equation*}%
where%
\begin{equation*}
V_{k}=\left( I_{2,k}\times (I_{4,k})^{n-1}\right) \cup \left( I_{3,k}\times
(I_{4,k})^{n-1}\right) ,\quad V_{i,k}=V_{i,k}^{1}\cup V_{i,k}^{2},\quad k\in 
\mathbb{Z},\quad i\in \{1,2,...,n-1\},
\end{equation*}%
with%
\begin{equation*}
V_{i,k}^{1}=(I_{1,k})^{n-i}\times I_{2,k}\times (I_{4,k})^{i-1}\quad \text{%
and\quad }V_{i,k}^{2}=(I_{1,k})^{n-i}\times I_{3,k}\times (I_{4,k})^{i-1}.
\end{equation*}%
If $i=1$, then we put $V_{1,k}^{1}=(I_{1,k})^{n-1}\times I_{2,k}\ $and\ $%
V_{1,k}^{2}=(I_{1,k})^{n-1}\times I_{3,k}$.

Let $x\in R_{k},k\in \mathbb{Z}$. Assume that $x$ does not belongs to the
set $J_{k}$. Then $x\notin V_{i,k}\ $and $x\notin V_{k}$ for any $i\in
\{1,2,...,n-1\}$. Since $x$ is not an element of $V_{1,k}^{1}\cup
V_{1,k}^{2} $, we have necessary that $(x_{1},...,x_{n-1})$ belongs in $%
(I_{1,k})^{n-1}$ and $x_{n}\in I_{4,k}$, otherwise $x$ is not an element of $%
R_{k}$, which is a contradiction. Assume that there exists $x_{i_{0}}\notin
I_{4,k}$ with $i_{0}\in \{2,...,n-1\}$. Observe that $x\notin
V_{n-i_{0}+1,k}^{1}\cup V_{n-i_{0}+1,k}^{2}$, which yields that%
\begin{equation*}
(x_{1},...,x_{i_{0}-1})\in (I_{1,k})^{i_{0}-1},\quad x_{i_{0}}\in
I_{2,k}\cup I_{3,k},\quad (x_{i_{0}+1},...,x_{n})\notin (I_{4,k})^{n-i_{0}}.
\end{equation*}%
Let 
\begin{equation*}
v=\max \big\{j:i_{0}\leq j<n,x_{j}\notin I_{4,k}\big\}.
\end{equation*}%
Hence 
\begin{equation}
x_{m}\in I_{4,k},\quad v+1\leq m<n.  \label{proof}
\end{equation}%
Also $x\notin V_{n-v+1,k}^{1}\cup V_{n-v+1,k}^{2}$, which yields that%
\begin{equation*}
(x_{1},...,x_{v-1})\in (I_{1,k})^{v-1},\quad x_{v}\in I_{2,k}\cup
I_{3,k},\quad (x_{v+1},...,x_{n})\notin (I_{4,k})^{n-v},
\end{equation*}%
which is a contradiction by \eqref{proof} and the fact that $x_{n}\in
I_{4,k} $. Consequently we obtain $x_{1}\in I_{1,k}$ and $%
(x_{2},...,x_{n})\in \left( I_{4,k}\right) ^{n-1}$. But $x\notin V_{k}$,
then we have $x_{1}\in I_{4,k}$, $x\in B(0,2^{k-1})$ and this is a
contradiction. Therefore%
\begin{equation*}
R_{k}\subset J_{k}\subset \tilde{R}_{k},\quad k\in \mathbb{Z},
\end{equation*}%
where $\tilde{R}_{k}=\{x\in \mathbb{R}^{n}:\frac{1}{\sqrt{n}}2^{k-3}\leq
|x|\leq \sqrt{n}2^{k+4}\}$. Let $f\in \mathcal{D}(\mathbb{R}^{n})$. We will
prove the inequality \eqref{KCN}. We write%
\begin{equation*}
\big\|f\big\|_{\dot{K}_{q}^{\alpha _{1},r}(\mathbb{R}^{n})}^{r}=\sum%
\limits_{k=-\infty }^{\infty }2^{k\alpha _{1}r}\big\|f\chi _{k}\big\|%
_{q}^{r}.
\end{equation*}%
Using H\"{o}lder's inequality we obtain%
\begin{equation*}
\big\|f\chi _{k}\big\|_{q}\leq c2^{k(\frac{n}{q}-n+1)}\big\|f\chi _{k}\big\|%
_{\frac{n}{n-1}},\quad k\in \mathbb{Z}\text{,}
\end{equation*}%
where the constant $c>0$ is independent of $k$. We have%
\begin{align*}
\int_{R_{k}}|f(x)|^{\frac{n}{n-1}}dx &\leq
\sum_{i=1}^{n-1}\int_{V_{i,k}}|f(x)|^{\frac{n}{n-1}}dx+\int_{V_{k}}|f(x)|^{%
\frac{n}{n-1}}dx \\
&\leq \sum_{i=1}^{n-1}\int_{V_{i,k}^{1}}|f(x)|^{\frac{n}{n-1}%
}dx+\sum_{i=1}^{n-1}\int_{V_{i,k}^{2}}|f(x)|^{\frac{n}{n-1}%
}dx+\int_{V_{k}}|f(x)|^{\frac{n}{n-1}}dx \\
&=\sum_{i=1}^{n-1}J_{i,k}^{1}+\sum_{i=1}^{n-1}J_{i,k}^{2}+S_{k.}.
\end{align*}%
We estimate $J_{i,k}^{1},i\in \{1,2,...,n-1\}$. Let $\omega _{1},\omega
_{2},\omega _{3}\in \mathcal{D}(\mathbb{R})$ be such that 
\begin{equation*}
\omega _{1}(y)=1\quad \text{if}\quad |y|\leq 2^{-M},\quad \omega
_{2}(y)=1\quad \text{if}\quad \frac{2^{-M-1}}{\sqrt{n}}\leq |y|\leq 2^{-M},
\end{equation*}%
\begin{equation*}
\omega _{3}(y)=1\quad \text{if}\quad |y|\leq \frac{2^{-M}}{2\sqrt{n}},
\end{equation*}
\begin{equation*}
\mathrm{supp}\omega _{1}\subset \{y\in \mathbb{R}:|y|\leq 2^{1-M}\},\quad 
\mathrm{supp}\omega _{2}\subset \{y\in \mathbb{R}:\frac{2^{-M-2}}{\sqrt{n}}%
\leq |y|\leq 2^{2-M}\}
\end{equation*}
and 
\begin{equation*}
\mathrm{supp}\omega _{3}\subset \{y\in \mathbb{R}:|y|\leq \frac{2^{1-M}}{%
\sqrt{n}}\},
\end{equation*}%
where $M>1$ will be chosen later on. Let 
\begin{equation*}
f_{k}(x)=f(x)\Pi _{j=1}^{n-i}\omega _{1}(2^{-k-M}x_{j})\omega
_{2}(2^{-k-M}x_{n-i+1})\Pi _{j=n-i+2}^{n}\omega _{3}(2^{-k-M}x_{j}),\quad
x\in \mathbb{R}^{n}.
\end{equation*}%
Obviously, if $x\in V_{i,k}^{1}$, then%
\begin{equation*}
f(x)=f_{k}(x).
\end{equation*}%
Let $x\in V_{i,k}^{1}$. Taking into account the various conditions on the
supports of $\omega _{1},\omega _{2}$ and $\omega _{3}$ we obtain%
\begin{equation*}
f(x)=\int_{-2^{k+1}}^{x_{j}}\frac{\partial f_{k}}{\partial x_{j}}%
(x_{1},...,x_{j-1},y_{j},x_{j+1},...,x_{n})dy_{j},
\end{equation*}%
which yields that%
\begin{equation*}
|f(x)|\leq \int_{-2^{k+1}}^{x_{j}}\Big|\frac{\partial f_{k}}{\partial x_{j}}%
(x_{1},...,x_{j-1},y_{j},x_{j+1},...,x_{n})\Big|dy_{j}
\end{equation*}%
for any $j\in \{1,2,...,n-i\}$. In the same way we obtain%
\begin{equation*}
|f(x)|\leq \int_{\frac{1}{6\sqrt{n}}2^{k}}^{x_{n-i+1}}\Big|\frac{\partial
f_{k}}{\partial x_{n-i+1}}(x_{1},...,x_{n-i},y_{n-i+1},x_{n-i+1},...,x_{n})%
\Big|dy_{n-i+1}
\end{equation*}%
and%
\begin{equation*}
|f(x)|\leq \int_{-\frac{1}{\sqrt{n}}2^{k+1}}^{x_{j}}\Big|\frac{\partial f_{k}%
}{\partial x_{j}}(x_{1},...,x_{n-i+1},...,x_{j-1},y_{j},x_{j+1},...,x_{n})%
\Big|dy_{j}
\end{equation*}%
for any $j\in \{n-i+2,...,n\}$. Therefore for any $x\in V_{i,k}^{1}$, $%
|f(x)|^{\frac{n}{n-1}}$ is bounded by%
\begin{align*}
&\prod_{j=1}^{n-i}\Big(\int_{-2^{k+1}}^{2^{k}}\Big|\frac{\partial f_{k}}{%
\partial x_{j}}(x_{1},...,x_{j-1},y_{j},x_{j+1},...,x_{n})\Big|dy_{j}\Big)^{%
\frac{1}{n-1}} \\
&\Big(\int_{\frac{1}{6\sqrt{n}}2^{k}}^{2^{k}}\Big|\frac{\partial f_{k}}{%
\partial x_{n-i+1}}(x_{1},...,x_{n-i},y_{n-i+1},x_{n-i+2},...,x_{n})\Big|%
dy_{n-i+1}\Big)^{\frac{1}{n-1}} \\
&\times \prod_{j=n-i+2}^{n}\Big(\int_{-\frac{1}{\sqrt{n}}2^{k+1}}^{\frac{%
2^{k-1}}{\sqrt{n}}}\Big|\frac{\partial f_{k}}{\partial x_{j}}%
(x_{1},...,x_{n-i+1},...,x_{j-1},y_{j},x_{j+1},...,x_{n})\Big|dy_{j}\Big)^{%
\frac{1}{n-1}} \\
&=\prod_{j=1}^{n-i}\big(g(x_{j}^{\prime })\big)^{\frac{1}{n-1}}\big(%
h(x_{n-i+1}^{\prime })\big)^{\frac{1}{n-1}}\prod_{j=n-i+2}^{n}\big(%
w(x_{j}^{\prime })\big)^{\frac{1}{n-1}},
\end{align*}%
where $x_{j}^{\prime }=(x_{1},...,x_{j-1},x_{j+1},...,x_{n}),j\in
\{1,...,n\} $ and $x_{n-i+1}^{\prime
}=(x_{1},...,x_{n-i},x_{n-i+2},...,x_{n})$. Integrate with respect to $x_{1}$%
, over $I_{1,k}$ to obtain $\int_{I_{1,k}}|f(x)|^{\frac{n}{n-1}}dx_{1}$ is
bounded by%
\begin{align*}
&\int_{I_{1,k}}\prod_{j=1}^{n-i}\big(g(x_{j}^{\prime })\big)^{\frac{1}{n-1}}%
\big(h(x_{n-i+1}^{\prime })\big)^{\frac{1}{n-1}}\prod_{j=n-i+2}^{n}\big(%
w(x_{j}^{\prime })\big)^{\frac{1}{n-1}}dx_{1} \\
&=\big(g(x_{1}^{\prime })\big)^{\frac{1}{n-1}}\int_{I_{1,k}}\prod_{j=2}^{n-i}%
\big(g(x_{j}^{\prime })\big)^{\frac{1}{n-1}}\big(h(x_{n-i+1}^{\prime })\big)%
^{\frac{1}{n-1}}\prod_{j=n-i+2}^{n}\big(w(x_{j}^{\prime })\big)^{\frac{1}{n-1%
}}dx_{1},
\end{align*}%
which is bounded by, after using H\"{o}lder's inequality,%
\begin{equation*}
\big(g(x_{1}^{\prime })\big)^{\frac{1}{n-1}}\prod_{j=2}^{n-i}\Big(%
\int_{I_{1,k}}g(x_{j}^{\prime })dx_{1}\Big)^{\frac{1}{n-1}}\Big(%
\int_{I_{1,k}}h(x_{n-i+1}^{\prime })dx_{1}\Big)^{\frac{1}{n-1}}\Big(%
\prod_{j=n-i+2}^{n}\int_{I_{1,k}}w_{j}(x_{j}^{\prime })dx_{1}\Big)^{\frac{1}{%
n-1}}.
\end{equation*}%
Integrate with respect to $x_{2}$, over $I_{1,k}$ and using H\"{o}lder's
inequality to obtain that $\int_{(I_{1,k})^{2}}|f(x)|^{\frac{n}{n-1}%
}dx_{1}dx_{2}$ is bounded by%
\begin{align*}
&\Big(\int_{I_{1,k}}g(x_{2}^{\prime })dx_{1}\Big)^{\frac{1}{n-1}%
}\int_{I_{1,k}}\big(g(x_{1}^{\prime })\big)^{\frac{1}{n-1}}\prod_{j=3}^{n-i}%
\Big(\int_{I_{1,k}}g(x_{j}^{\prime })dx_{1}\Big)^{\frac{1}{n-1}} \\
&\times \Big(\int_{I_{1,k}}h(x_{n-i+1}^{\prime })dx_{1}\Big)^{\frac{1}{n-1}}%
\Big(\prod_{j=n-i+2}^{n}\int_{I_{1,k}}w(x_{j}^{\prime })dx_{1}\Big)^{\frac{1%
}{n-1}}dx_{2} \\
&\leq \Big(\int_{I_{1,k}}g(x_{2}^{\prime })dx_{1}\Big)^{\frac{1}{n-1}}\Big(%
\int_{I_{1,k}}g(x_{1}^{\prime })dx_{2}\Big)^{\frac{1}{n-1}}\prod_{j=3}^{n-i}%
\Big(\int_{(I_{1,k})^{2}}g(x_{j}^{\prime })dx_{1}dx_{2}\Big)^{\frac{1}{n-1}}
\\
&\times \Big(\int_{(I_{1,k})^{2}}h(x_{n-i+1}^{\prime })dx_{1}dx_{2}\Big)^{%
\frac{1}{n-1}}\Big(\prod_{j=n-i+2}^{n}\int_{(I_{1,k})^{2}}w(x_{j}^{\prime
})dx_{1}dx_{2}\Big)^{\frac{1}{n-1}}.
\end{align*}%
Hence $\int_{(I_{1,k})^{n-i}}|f(x)|^{\frac{n}{n-1}}dx_{1}\cdot \cdot \cdot
dx_{n-i}$ is bounded by%
\begin{align*}
&\prod_{j=1}^{n-i}\Big(\int_{(I_{1,k})^{n-i-1}}g(x_{j}^{\prime })dx_{1}\cdot
\cdot \cdot dx_{j-1}dx_{j+1}\cdot \cdot \cdot dx_{n-i-1}\Big)^{\frac{1}{n-1}}
\\
&\times \Big(\int_{(I_{1,k})^{n-i}}h(x_{n-i+1}^{\prime })dx_{1}dx_{2}\cdot
\cdot \cdot dx_{n-i}\Big)^{\frac{1}{n-1}} \\
&\times \Big(\prod_{j=n-i+2}^{n}\int_{(I_{1,k})^{n-i}}w(x_{j}^{\prime
})dx_{1}dx_{2}\cdot \cdot \cdot dx_{n-i}\Big)^{\frac{1}{n-1}}.
\end{align*}%
In the same way $\int_{(I_{1,k})^{n-i}\times I_{2,k}}|f(x)|^{\frac{n}{n-1}%
}dx_{1}\cdot \cdot \cdot dx_{n-i}dx_{n-i+1}$ is bounded by%
\begin{align*}
&\prod_{j=1}^{n-i}\Big(\int_{(I_{1,k})^{n-i-1}\times I_{2,k}}g(x_{j}^{\prime
})dx_{1}\cdot \cdot \cdot dx_{j-1}dx_{j+1}\cdot \cdot \cdot dx_{n-i+1}\Big)^{%
\frac{1}{n-1}} \\
&\times \Big(\int_{(I_{1,k})^{n-i}}h(x_{n-i+1}^{\prime })dx_{1}dx_{2}\cdot
\cdot \cdot dx_{n-i}\Big)^{\frac{1}{n-1}} \\
&\times \prod_{j=n-i+2}^{n}\Big(\int_{(I_{1,k})^{n-i}\times
I_{2,k}}w(x_{j}^{\prime })dx_{1}dx_{2}\cdot \cdot \cdot dx_{n-i+1}\Big)^{%
\frac{1}{n-1}}.
\end{align*}%
Consequently $\int_{V_{i,k}^{1}}|f(x)|^{\frac{n}{n-1}}dx_{1}\cdot \cdot
\cdot dx_{n}$ is bounded by%
\begin{align*}
&\prod_{j=1}^{n-i}\Big(\int_{(I_{1,k})^{n-i-1}\times I_{2,k}\times
(I_{4,k})^{i-1}}g(x_{j}^{\prime })dx_{1}\cdot \cdot \cdot
dx_{j-1}dx_{j+1}\cdot \cdot \cdot dx_{n}\Big)^{\frac{1}{n-1}} \\
&\times \Big(\int_{(I_{1,k})^{n-i}\times (I_{4,k})^{i-1}}h(x_{n-i+1}^{\prime
})dx_{1}dx_{2}\cdot \cdot \cdot dx_{n-i}\Big)^{\frac{1}{n-1}} \\
&\times \prod_{j=n-i+2}^{n}\Big(\int_{(I_{1,k})^{n-i}\times I_{2,k}\times
(I_{4,k})^{i-2}}w(x_{j}^{\prime })dx_{1}dx_{2}\cdot \cdot \cdot \cdot
dx_{j-1}dx_{j+1}\cdot \cdot \cdot dx_{n}\Big)^{\frac{1}{n-1}},
\end{align*}%
which is bounded by%
\begin{equation*}
\prod_{j=1}^{n}\Big(\int_{\tilde{R}_{k}}\Big|\frac{\partial f_{k}}{\partial
x_{j}}(x)\Big|dx\Big)^{\frac{1}{n-1}}.
\end{equation*}%
Observe that%
\begin{equation*}
\Big|\frac{\partial f_{k}}{\partial x_{j}}\Big|\leq C2^{-(k+M)}\big|f\big|+%
\big|\nabla f\big|,\quad j\in \{1,2,...,n\},
\end{equation*}%
where the positive constant $C$ is independent of $k$. Consequently,%
\begin{equation*}
J_{i,k}^{1}\leq \Big(C2^{-(k+M)}\big\|f\chi _{\tilde{R}_{k}}\big\|_{L^{1}(%
\mathbb{R}^{n})}+\big\|(\nabla f)\chi _{\tilde{R}_{k}}\big\|_{L^{1}(\mathbb{R%
}^{n})}\Big)^{\frac{n}{n-1}}
\end{equation*}%
for any $k\in \mathbb{Z}$ and any $i\in \{1,2,...,n-1\}$. Using the fact
that $\alpha _{2}+n-1=\alpha _{1}+\frac{n}{q}$ we deduce the following
estimation%
\begin{align*}
&\Big(\sum\limits_{k=-\infty }^{\infty }2^{k(\alpha _{1}+\frac{n}{q}-n+1)r}%
\Big(\sum_{i=1}^{n-1}J_{i,k}^{1}\Big)^{\frac{(n-1)r}{n}}\Big)^{\frac{1}{r}}
\\
&\leq c_{1}\big\|f\big\|_{\dot{W}_{1,1}^{\alpha _{2},r}(\mathbb{R}%
^{n})}+c_{2}2^{-M}\big\|f\big\|_{\dot{K}_{1}^{\alpha _{2}-1,r}(\mathbb{R}%
^{n})},
\end{align*}%
where both $c_{1}$ and $c_{2}$ are independent of $M$. We estimate $%
V_{i,k}^{2},i\in \{1,2,...,n-1\}$. We have%
\begin{equation*}
J_{i,k}^{2}=\int_{(I_{1,k})^{n-i}\times I_{2,k}\times (I_{4,k})^{i-1}}\big|%
f_{k}(x_{1},...,x_{n-i},-x_{n-i+1},x_{n-i+2},...,x_{n})\big|^{\frac{n}{n-1}%
}dx_{1}dx_{2}\cdot \cdot \cdot dx_{n}
\end{equation*}%
for any $k\in \mathbb{Z}$ and any $i\in \{1,2,...,n-1\}$. The estimate of $%
\sum_{i=1}^{n-1}J_{i,k}^{2}$ can be done in the same way as in $V_{i,k}^{1}$%
. The estimate of $S_{k.}$ can be done in the same way as in $%
\sum_{i=1}^{n-1}J_{i,k}^{1}$ and $\sum_{i=1}^{n-1}J_{i,k}^{2}$. Collecting
these estimations in one formula we find that 
\begin{equation*}
\big\|f\big\|_{\dot{K}_{q}^{\alpha _{1},r}(\mathbb{R}^{n})}\leq c_{3}\big\|f%
\big\|_{\dot{W}_{1,1}^{\alpha _{2},r}(\mathbb{R}^{n})}+c_{4}2^{-M}\big\|f%
\big\|_{\dot{K}_{1}^{\alpha _{2}-1,r}(\mathbb{R}^{n})},
\end{equation*}%
where both $c_{3}$ and $c_{4}$ are independent of $M$. Using the fact that $%
\alpha _{2}+n-1=\alpha _{1}+\frac{n}{q}$\ and H\"{o}lder's inequality to
obtain%
\begin{equation*}
\big\|f\big\|_{\dot{K}_{1}^{\alpha _{2}-1,r}(\mathbb{R}^{n})}\leq B\big\|f%
\big\|_{\dot{K}_{q}^{\alpha _{1},r}(\mathbb{R}^{n})}.
\end{equation*}%
Choosing $M$ such that $c_{4}B2^{-M}\leq \frac{1}{2}$ we obtain the desired
inequality.
\end{proof}

\begin{remark}
We mention here that our embedding covers the Caffarelli--Kohn--Nirenberg
inequality because of \eqref{CKN-condition} yields that $1\leq q\leq \frac{n%
}{n-1}$.
\end{remark}

\begin{theorem}
\label{Embeddings2}Let $1\leq q\leq \frac{n}{\frac{n}{p}-1},0<r\leq \infty
,\alpha _{2}\geq \alpha _{1}$, and%
\begin{equation*}
\frac{n}{p}+\alpha _{2}-1=\frac{n}{q}+\alpha _{1}>0.
\end{equation*}%
Then%
\begin{equation}
\big\|f\big\|_{\dot{K}_{q}^{\alpha _{1},r}(\mathbb{R}^{n})}\lesssim \big\|f%
\big\|_{\dot{W}_{p,1}^{\alpha _{2},r}(\mathbb{R}^{n})},\quad f\in \mathcal{D}%
(\mathbb{R}^{n}),  \label{embeddings1}
\end{equation}%
holds, where%
\begin{equation*}
\big\|f\big\|_{\dot{W}_{p,1}^{\alpha _{2},r}(\mathbb{R}^{n})}=\Big(%
\sum\limits_{k=-\infty }^{\infty }2^{k\alpha _{2}r}\big\|(\nabla f)\chi _{k}%
\big\|_{p}^{r}\Big)^{\frac{1}{r}}.
\end{equation*}
\end{theorem}

\begin{proof}
Let $f\in \mathcal{D}(\mathbb{R}^{n})$ and\ $\frac{n}{\sigma }=\frac{n}{q}+n-%
\frac{n}{p}$. According to Theorem \ref{Embeddings1}, since $1\leq \sigma
\leq \frac{n}{n-1}$, one has%
\begin{equation}
\big\|g\big\|_{\dot{K}_{\sigma }^{\alpha _{1},\tau }(\mathbb{R}%
^{n})}\lesssim \sum_{j=1}^{n}\Big\|\frac{\partial g}{\partial x_{j}}\Big\|_{%
\dot{K}_{1}^{\alpha _{2},\tau }(\mathbb{R}^{n})},\quad 0<\tau \leq \infty .
\label{inequality1}
\end{equation}%
Let $g=|f|^{\frac{q}{\sigma }}$. It is easily seen that%
\begin{equation*}
\big\|f\big\|_{\dot{K}_{q}^{\alpha _{1},r}(\mathbb{R}^{n})}=\big\|g\big\|_{%
\dot{K}_{\sigma }^{\frac{q}{\sigma }\alpha _{1},r\frac{\sigma }{q}}(\mathbb{R%
}^{n})}^{\frac{\sigma }{q}}.
\end{equation*}%
Let 
\begin{equation*}
\widetilde{\alpha _{2}}=\alpha _{2}-\alpha _{1}+\frac{q}{\sigma }\alpha
_{1}\quad \text{and}\quad \widetilde{r}=r\frac{\sigma }{q}.
\end{equation*}%
From the inequality \eqref{inequality1}, we deduce 
\begin{align*}
\big\|g\big\|_{\dot{K}_{\sigma }^{\frac{q}{\sigma }\alpha _{1},\widetilde{r}%
}(\mathbb{R}^{n})} &\lesssim \sum_{j=1}^{n}\Big\|\frac{\partial g}{\partial
x_{j}}\Big\|_{\dot{K}_{1}^{\widetilde{\alpha _{2}},\widetilde{r}}(\mathbb{R}%
^{n})} \\
&\lesssim \sum_{j=1}^{n}\Big\|s|f|^{s-1}\frac{\partial f}{\partial x_{j}}%
\Big\|_{\dot{K}_{1}^{\widetilde{\alpha _{2}},\widetilde{r}}(\mathbb{R}^{n})},
\end{align*}%
with $s=\frac{q}{\sigma }$. By combining this estimate with%
\begin{equation*}
\widetilde{\alpha _{2}}=\alpha _{1}\frac{q}{p^{\prime }}+\alpha _{2}\quad 
\text{and}\quad \frac{1}{\widetilde{r}}=\frac{1}{r}\frac{q}{p^{\prime }}+%
\frac{1}{r}
\end{equation*}%
we see that%
\begin{align*}
\Big\|s|f|^{s-1}\frac{\partial f}{\partial x_{j}}\Big\|_{\dot{K}_{1}^{%
\widetilde{\alpha _{2}},\widetilde{r}}(\mathbb{R}^{n})} &\lesssim \big\|%
s|f|^{s-1}\big\|_{\dot{K}_{p^{\prime }}^{\alpha _{1}\frac{q}{p^{\prime }},%
\frac{rp^{\prime }}{q}}(\mathbb{R}^{n})}\Big\|\frac{\partial f}{\partial
x_{j}}\Big\|_{\dot{K}_{p}^{\alpha _{2},r}(\mathbb{R}^{n})} \\
&=c\big\|f\big\|_{\dot{K}_{q}^{\alpha _{1},r}(\mathbb{R}^{n})}^{\frac{q}{%
p^{\prime }}}\Big\|\frac{\partial f}{\partial x_{j}}\Big\|_{\dot{K}%
_{p}^{\alpha _{2},r}(\mathbb{R}^{n})},
\end{align*}%
where we have used the H\"{o}lder inequality. Therefore%
\begin{equation*}
\big\|f\big\|_{\dot{K}_{q}^{\alpha _{1},r}(\mathbb{R}^{n})}\lesssim \big\|f%
\big\|_{\dot{K}_{q}^{\alpha _{1},r}(\mathbb{R}^{n})}^{\frac{\sigma }{%
p^{\prime }}}\sum_{j=1}^{n}\Big\|\frac{\partial f}{\partial x_{j}}\Big\|_{%
\dot{K}_{p}^{\alpha _{2},r}(\mathbb{R}^{n})}^{\frac{\sigma }{q}}
\end{equation*}%
and get finally%
\begin{equation*}
\big\|f\big\|_{\dot{K}_{q}^{\alpha _{1},r}(\mathbb{R}^{n})}\lesssim
\sum_{j=1}^{n}\Big\|\frac{\partial f}{\partial x_{j}}\Big\|_{\dot{K}%
_{p}^{\alpha _{2},r}(\mathbb{R}^{n})}.
\end{equation*}%
Hence the proof is complete.
\end{proof}

\begin{remark}
Again our embedding covers the Caffarelli--Kohn--Nirenberg inequality
because of \eqref{CKN-condition} yields that $1\leq q\leq \frac{n}{\frac{n}{p%
}-1}$. Let $1\leq p\leq q<\infty $ and $\frac{n}{q}-\frac{n}{p}=\alpha
_{2}-1-\alpha _{1}$. By \eqref{embeddings1}\ we easily obtain that%
\begin{equation*}
\Big(\int_{\mathbb{R}^{n}}|x|^{\alpha _{1}q}|f(x)|^{q}dx\Big)^{\frac{1}{q}%
}\lesssim \Big(\sum\limits_{k=-\infty }^{\infty }2^{k\alpha _{2}q}\big\|%
(\nabla f)\chi _{k}\big\|_{p}^{q}\Big)^{\frac{1}{q}}\lesssim \Big(\int_{%
\mathbb{R}^{n}}|x|^{\alpha _{2}p}|\nabla f(x)|^{p}dx\Big)^{\frac{1}{p}},
\end{equation*}%
whenever the right-hand side is finite. In particular,%
\begin{equation*}
\big\|f\big\|_{q}\lesssim \Big(\sum\limits_{k=-\infty }^{\infty }\big\|%
(\nabla f)\chi _{k}\big\|_{p}^{q}\Big)^{\frac{1}{q}}\lesssim \big\|\nabla f%
\big\|_{p},
\end{equation*}%
where $1\leq p<q<\infty $ and $1-\frac{n}{p}=-\frac{n}{q}$, whenever the
right-hand side is finite, which is the Sobolev's inequality.
\end{remark}

In reality, the inequality of Caffarelli--Kohn--Nirenberg inequality says
that%
\begin{equation}
\big\||x|^{\alpha _{1}}f\big\|_{q}\leq c\big\||x|^{\alpha _{2}}f\big\|%
_{p}^{\theta }\big\||x|^{\alpha _{3}}\nabla f\big\|_{u}^{1-\theta },\quad
f\in \mathcal{D}(\mathbb{R}^{n}),  \label{CKN}
\end{equation}%
where $p,u\geq 1,q>0,0\leq \theta \leq 1,$%
\begin{equation*}
\frac{n}{p}+\alpha _{2}>0,\quad \frac{n}{u}+\alpha _{3}>0,\quad \frac{n}{q}%
+\alpha _{1}>0,
\end{equation*}%
\begin{equation*}
\frac{n}{q}+\alpha _{1}=\theta \big(\frac{n}{p}+\alpha _{2}-1\big)+\big(%
\frac{n}{u}+\alpha _{3}\big)(1-\theta ),\quad \alpha _{1}=\theta \sigma
+(1-\theta )\alpha _{3},
\end{equation*}%
\begin{equation*}
\sigma \leq \alpha _{2}\quad \text{if}\quad \theta >0
\end{equation*}%
and%
\begin{equation*}
\alpha _{2}\leq \sigma +1\quad \text{if}\quad \theta >0\text{\quad and\quad }%
\frac{n}{q}+\alpha _{1}=\frac{n}{p}+\alpha _{2}-1.
\end{equation*}%
Our aim is to extend this result to Herz spaces. We begin by the following
special case.

\begin{theorem}
\label{Embeddings3}Let $u\geq 1,q,v,r,s>0,0\leq \theta \leq 1,$%
\begin{equation*}
n+\alpha _{2}>0,\quad \frac{n}{u}+\alpha _{3}>0,\quad \frac{n}{q}+\alpha
_{1}>0,\quad \sigma \leq \alpha _{2}\leq \sigma +1,
\end{equation*}%
\begin{equation*}
\alpha _{1}=\theta \sigma +(1-\theta )\alpha _{3},\quad \frac{n}{q}+\alpha
_{1}=\theta \big(n+\alpha _{2}-1\big)+\big(\frac{n}{u}+\alpha _{3}\big)%
(1-\theta )
\end{equation*}%
and%
\begin{equation*}
\frac{1}{r}=\frac{\theta }{s}+\frac{1-\theta }{v}.
\end{equation*}%
Then%
\begin{equation*}
\big\|f\big\|_{\dot{K}_{q}^{\alpha _{1},r}(\mathbb{R}^{n})}\leq c\big\|%
\nabla f\big\|_{\dot{K}_{1}^{\alpha _{2},s}(\mathbb{R}^{n})}^{\theta }\big\|f%
\big\|_{\dot{K}_{u}^{\alpha _{3},v}(\mathbb{R}^{n})}^{1-\theta },\quad f\in 
\mathcal{D}(\mathbb{R}^{n}),
\end{equation*}
\end{theorem}

\begin{proof}
Obviously, we need only to study the case $0<\theta <1$. Let $h=\frac{n}{%
n-1+\alpha _{2}-\sigma }$. Therefore%
\begin{equation*}
\frac{1}{q}=\frac{\theta }{h}+\frac{1-\theta }{u}.
\end{equation*}%
Using H\"{o}lder's inequality we obtain%
\begin{equation*}
2^{k\alpha _{1}}\big\|f\chi _{k}\big\|_{q}\leq c\big(2^{k\sigma }\big\|f\chi
_{k}\big\|_{h}\big)^{\theta }\big(2^{k\alpha _{3}}\big\|f\chi _{k}\big\|_{u}%
\big)^{1-\theta },\quad k\in \mathbb{Z}\text{.}
\end{equation*}%
Therefore%
\begin{equation*}
\big\|f\big\|_{\dot{K}_{q}^{\alpha _{1},r}(\mathbb{R}^{n})}\leq c\big\|f%
\big\|_{\dot{K}_{h}^{\sigma ,s}(\mathbb{R}^{n})}^{\theta }\big\|f\big\|_{%
\dot{K}_{u}^{\alpha _{3},v}(\mathbb{R}^{n})}^{1-\theta }.
\end{equation*}%
Observe that 
\begin{equation*}
\frac{n}{h}+\sigma =n-1+\alpha _{2},\quad 1\leq h\leq \frac{n}{n-1}.
\end{equation*}%
Hence by Theorem \ref{Embeddings1},%
\begin{equation*}
\big\|f\big\|_{\dot{K}_{h}^{\sigma ,s}(\mathbb{R}^{n})}\leq c\big\|\nabla f%
\big\|_{\dot{K}_{1}^{\alpha _{2},s}(\mathbb{R}^{n})}.
\end{equation*}%
The proof is complete.
\end{proof}

Now we formulate our main theorem.

\begin{theorem}
\label{Embeddings4}Let $p,u\geq 1,q,r,v,s>0,0\leq \theta \leq 1,$%
\begin{equation*}
\frac{n}{p}+\alpha _{2}>0,\quad \frac{n}{u}+\alpha _{3}>0,\quad \frac{n}{q}%
+\alpha _{1}>0,\quad \sigma \leq \alpha _{2}\leq \sigma +1,
\end{equation*}%
\begin{equation*}
\alpha _{1}=\theta \sigma +(1-\theta )\alpha _{3},\quad \frac{n}{q}+\alpha
_{1}=\theta \big(\frac{n}{p}+\alpha _{2}-1\big)+\big(\frac{n}{u}+\alpha _{3}%
\big)(1-\theta )
\end{equation*}%
and%
\begin{equation*}
\frac{1}{r}=\frac{\theta }{s}+\frac{1-\theta }{v}.
\end{equation*}%
Then%
\begin{equation}
\big\|f\big\|_{\dot{K}_{q}^{\alpha _{1},r}(\mathbb{R}^{n})}\leq c\big\|%
\nabla f\big\|_{\dot{K}_{p}^{\alpha _{2},s}(\mathbb{R}^{n})}^{\theta }\big\|f%
\big\|_{\dot{K}_{u}^{\alpha _{3},v}(\mathbb{R}^{n})}^{1-\theta },\quad f\in 
\mathcal{D}(\mathbb{R}^{n}),  \label{CKN-Inequality}
\end{equation}
\end{theorem}

\begin{proof}
We have $\frac{1}{q}=\frac{\theta }{\tau }+\frac{1-\theta }{u}$, where $\tau
=\frac{n}{\frac{n}{p}-1+\alpha _{2}-\sigma }$. Using H\"{o}lder's inequality
we obtain%
\begin{equation*}
\big\|f\big\|_{\dot{K}_{q}^{\alpha _{1},r}(\mathbb{R}^{n})}\leq c\big\|f%
\big\|_{\dot{K}_{\tau }^{\sigma ,s}(\mathbb{R}^{n})}^{\theta }\big\|f\big\|_{%
\dot{K}_{u}^{\alpha _{3},v}(\mathbb{R}^{n})}^{1-\theta }.
\end{equation*}%
Observe that%
\begin{equation*}
\frac{n}{\tau }-\frac{n}{p}=\alpha _{2}-1-\sigma \leq 0.
\end{equation*}%
According to Theorem \ref{Embeddings2}, since $1\leq \tau \leq \frac{n}{%
\frac{n}{p}-1}$, one has%
\begin{equation*}
\big\|f\big\|_{\dot{K}_{\tau }^{\sigma ,s}(\mathbb{R}^{n})}\leq c\big\|%
\nabla f\big\|_{\dot{K}_{p}^{\alpha _{2},s}(\mathbb{R}^{n})},
\end{equation*}%
which completes the proof..
\end{proof}

\begin{remark}
More Caffarelli--Kohn--Nirenberg inequalities in function spaces are given
in \cite{Drihem20}. From \eqref{CKN-Inequality} we easily obtain%
\begin{equation*}
\big\|f\big\|_{\dot{K}_{q}^{\alpha _{1},q}(\mathbb{R}^{n})}\leq c\big\|%
\nabla f\big\|_{\dot{K}_{p}^{\alpha _{2},\tau }(\mathbb{R}^{n})}^{\theta }%
\big\|f\big\|_{\dot{K}_{u}^{\alpha _{3},u}(\mathbb{R}^{n})}^{1-\theta
},\quad f\in \mathcal{D}(\mathbb{R}^{n}),
\end{equation*}%
but $\tau \geq p$, then we obtain 
\begin{equation*}
\big\|f\big\|_{\dot{K}_{q}^{\alpha _{1},q}(\mathbb{R}^{n})}\leq c\big\|%
\nabla f\big\|_{\dot{K}_{p}^{\alpha _{2},p}(\mathbb{R}^{n})}^{\theta }\big\|f%
\big\|_{\dot{K}_{u}^{\alpha _{3},u}(\mathbb{R}^{n})}^{1-\theta },\quad f\in 
\mathcal{D}(\mathbb{R}^{n}),
\end{equation*}%
which is the classical Caffarelli--Kohn--Nirenberg inequality, see %
\eqref{CKN}.
\end{remark}

\section{Herz-type Sobolev spaces}

In this section we prove the basic properties of Herz-type Sobolev spaces in
analogy to the classical Sobolev spaces.

\begin{definition}
Let $\Omega \subset \mathbb{R}^{n}$ be open, $(\alpha ,p,q)\in V_{\alpha
,p,q}$ and $m\in \mathbb{N}_{0}$. We define the Herz-type Sobolev space $%
\dot{K}_{p,m}^{\alpha ,q}(\Omega )$ as the set of functions $f\in \dot{K}%
_{p}^{\alpha ,q}(\Omega )$ with weak derivatives $D^{\beta }f\in \dot{K}%
_{p}^{\alpha ,q}(\Omega )$ for $|\beta |\leq m$. We define the norm of $\dot{%
K}_{p,m}^{\alpha ,q}(\Omega )$ by%
\begin{equation*}
\big\|f\big\|_{\dot{K}_{p,m}^{\alpha ,q}(\Omega )}=\Big(\sum\limits_{k=-%
\infty }^{\infty }2^{k\alpha q}\Big(\sum_{|\beta |\leq m}\big\|(D^{\beta
}f)\chi _{R_{k}\cap \Omega }\big\|_{p}^{p}\Big)^{\frac{q}{p}}\Big)^{1/q}
\end{equation*}%
if $1\leq p,q<\infty $ and%
\begin{equation*}
\big\|f\big\|_{\dot{K}_{p,m}^{\alpha ,\infty }(\Omega )}=\sup_{k\in \mathbb{Z%
}}2^{k\alpha }\Big(\sum_{|\beta |\leq m}\big\|(D^{\beta }f)\chi _{R_{k}\cap
\Omega }\big\|_{p}^{p}\Big)^{\frac{1}{p}}.
\end{equation*}
\end{definition}

\begin{remark}
One recognizes immediately that if $p=q$ and $\alpha =0$, then $\dot{K}%
_{p,m}^{0,p}(\Omega )=W_{p}^{m}(\Omega )$.
\end{remark}

As in classical Sobolev spaces, see \cite[Theorem 3.3]{AdamsFournier03}, we
have the following statements:

\begin{theorem}
Let $\Omega \subset \mathbb{R}^{n}$\ be open and $(\alpha ,p,q)\in V_{\alpha
,p,q}$.\ For each $m\in \mathbb{N}_{0}$, the Herz-type Sobolev space $\dot{K}%
_{p,m}^{\alpha ,q}(\Omega )$ is a Banach space.
\end{theorem}

Exactly in the same way as in the classical Sobolev spaces, see \cite%
{AdamsFournier03}, but we use Theorem \ref{approximation1} we immediately
arrive at the following result.

\begin{lemma}
Let $\Omega \subset \mathbb{R}^{n}$ be open, $m\in \mathbb{N}_{0}$ and $%
(\alpha ,p,q)\in V_{\alpha ,p,q}$ with $1<p<\infty ,1\leq q<\infty $ and $-%
\frac{n}{p}<\alpha <n-\frac{n}{p}$. Let $\Omega ^{\prime }$ be an open
subset of $\Omega $ such that $\overline{\Omega ^{\prime }}$ is a compact
subset of $\Omega $. Let $J_{\varepsilon }$ be as above and $f\in \dot{K}%
_{p,m}^{\alpha ,q}(\Omega )$. Then%
\begin{equation*}
\lim_{\varepsilon \rightarrow 0_{+}}\big\|J_{\varepsilon }\ast f-f\big\|_{%
\dot{K}_{p,m}^{\alpha ,q}(\Omega ^{\prime })}=0.
\end{equation*}
\end{lemma}

Similarly as in \cite[Theorems 3.6 and 3.17]{AdamsFournier03} with the help
of Theorem \ref{K-separable} we have the following statements:

\begin{theorem}
\label{density}Let $\Omega \subset \mathbb{R}^{n}$ be open, $m\in \mathbb{N}%
_{0}$ and $(\alpha ,p,q)\in V_{\alpha ,p,q}$ with $1<p<\infty ,1\leq
q<\infty $ and $-\frac{n}{p}<\alpha <n-\frac{n}{p}$. $\dot{K}_{p,m}^{\alpha
,q}(\Omega )$ is separable and $C^{\infty }(\Omega )\cap \dot{K}%
_{p,m}^{\alpha ,q}(\Omega )$ is dense in $\dot{K}_{p,m}^{\alpha ,q}(\Omega
). $
\end{theorem}

\subsection{Embeddings}

In this subsection we present some embeddings of the spaces introduced above.

\begin{definition}
Let $v\in \mathbb{R}^{n}\backslash \{0\}$ and for each $x\neq 0$ let $\angle
(x,v)$ be the angle between the position vector $x$ and $v$. Let $\kappa $
satisfying $0<\kappa <\pi $. The set 
\begin{equation*}
C=\{x\in \mathbb{R}^{n}:x=0\text{ or }0<|x|\leq \varrho ,\angle (x,v)\leq
\kappa /2\}
\end{equation*}%
is called a finite cone of height $\varrho $, axis direction $v$ and
aperture angle $\kappa $ with vertex at the origin.
\end{definition}

\begin{remark}
Let $C$ be a finite cone with vertex at the origin. Note that $%
x+C=\{x+y:y\in C\}$ is a finite cone with vertex at $x$ but the same
dimensions and axis direction as $C$ and is obtained by parallel translation
of $C$.
\end{remark}

We are now in a position to state the definition of domain satisfying the
cone condition

\begin{definition}
\label{conecondition}Let $\Omega \subset \mathbb{R}^{n}$ be open. $\Omega $
satisfies the cone condition if there exists a finite cone $C$ such that
each $x\in \Omega $ is the vertex of a finite cone $C_{x}$ contained in $%
\Omega $ and congruent to $C$.
\end{definition}

\begin{remark}
In Definition \ref{conecondition} the cone $C_{x}$ is not obtained from $C$
by parallel translation, but simply by rigid motion.
\end{remark}

The following statement can be found in \cite[Lemma 4.15]{AdamsFournier03},
that plays an essential for us.

\begin{lemma}
\label{Sobolev-rep1}Let $\Omega \subset \mathbb{R}^{n}$ be a domain
satisfying the cone condition. Then we can find a positive constant $K$
depending on $m,n$, and the dimensions $\varrho $ and $\kappa $ of the cone $%
C$ specified the cone condition for $\Omega $ such that for every $f\in
C^{\infty }(\Omega )$, every $x\in \Omega $, and every $r$ satisfying $%
0<r\leq \varrho $, we have%
\begin{equation*}
|f(x)|\leq K\Big(\sum_{|\beta |\leq m-1}r^{|\beta
|-n}\int_{C_{x,r}}|D^{\beta }f(y)|dy+\sum_{|\beta |=m-1}\int_{C_{x,r}}\frac{%
|D^{\beta }f(y)|}{|x-y|^{n-m}}dy\Big),
\end{equation*}%
where $C_{x,r}=\{y\in C_{x}:y\in B(x,r)\}.$
\end{lemma}

Let $0<\lambda <n$. The Riesz potential operator $\mathcal{I}_{\lambda }$ is
defined by 
\begin{equation*}
\mathcal{I}_{\lambda }f(x)=\int_{{\mathbb{R}^{n}}}\frac{f(y)}{%
|x-y|^{n-\lambda }}dy.
\end{equation*}%
Let $p^{\ast }$ be the Sobolev exponent defined by $\frac{1}{p^{\ast }}=%
\frac{1}{p}-\frac{\lambda }{n}.$ The following\ statement plays a crucial
role in our embeddings\ results, see \cite{LiYang96}.

\begin{theorem}
\label{result3}Let $0<\lambda <n$, $0<q_{0}\leq q_{1}\leq \infty $ and $%
1<p<p^{\ast }<\frac{n}{\lambda }$. If 
\begin{equation*}
\lambda -\frac{n}{p}<\alpha <n-\frac{n}{p},
\end{equation*}%
then $\mathcal{I}_{\lambda }$ is bounded from $\dot{K}_{{p}}^{{\alpha ,}%
q_{0}}({\mathbb{R}^{n}})$ into $\dot{K}_{p^{\ast }}^{{\alpha ,}q_{1}}({%
\mathbb{R}^{n}})$.
\end{theorem}

Now we state the first embeddings theorem.

\begin{theorem}
\label{embeddingsfirst}Let $\Omega \subset \mathbb{R}^{n}$ be a domain
satisfying the cone condition, $0\in \Omega $ and $m\in \mathbb{N}_{0}$. Let 
$1<p<\infty ,1\leq r<\infty ,\alpha _{2}\geq \alpha _{1},m-\frac{n}{p}%
<\alpha _{2}<n-\frac{n}{p},$%
\begin{equation}
m-\alpha _{2}+\alpha _{1}>0\quad \text{and}\quad \frac{n}{q}=\frac{n}{p}%
-m+\alpha _{2}-\alpha _{1}>0.  \label{sobolevassumption1}
\end{equation}%
Then%
\begin{equation*}
\dot{K}_{p,m}^{\alpha _{2},r}(\Omega )\hookrightarrow \dot{K}_{q}^{\alpha
_{1},r}(\Omega )
\end{equation*}%
holds.{}
\end{theorem}

\begin{proof}
We use Theorem \ref{density} and we will do the proof in two steps. Let $%
f\in C^{\infty }(\Omega )\cap \dot{K}_{p,m}^{\alpha _{2},r}(\Omega )$.

\textbf{Step 1.}\textit{\ }$\alpha _{1}=\alpha _{2}$. From Lemma \ref%
{Sobolev-rep1},%
\begin{equation*}
|f(x)|\lesssim \sum_{|\beta |\leq m}\mathcal{I}_{m}((D^{\beta }f)\chi
_{\Omega })(x),\quad x\in \Omega .
\end{equation*}%
Using Theorem \ref{result3} we obtain%
\begin{equation*}
\big\|f\big\|_{\dot{K}_{q}^{\alpha _{1},r}(\Omega )}\lesssim \sum_{|\beta
|\leq m}\big\|(D^{\beta }f)\chi _{\Omega }\big\|_{\dot{K}_{p}^{\alpha
_{1},r}(\mathbb{R}^{n})}\lesssim \big\|f\big\|_{\dot{K}_{p,m}^{\alpha
_{1},r}(\Omega )}.
\end{equation*}

\textbf{Step 2.} $\alpha _{2}>\alpha _{1}$. We write%
\begin{align*}
\big\|f\big\|_{\dot{K}_{q}^{\alpha _{1},r}(\Omega )}^{r}
&=\sum\limits_{k=-\infty }^{\infty }2^{k\alpha _{1}r}\big\|f\chi _{R_{k}\cap
\Omega }\big\|_{q}^{r} \\
&=\sum\limits_{k=-\infty }^{-1}2^{k\alpha _{1}r}\big\|f\chi _{R_{k}\cap
\Omega }\big\|_{q}^{r}+\sum\limits_{k=0}^{\infty }2^{k\alpha _{1}r}\big\|%
f\chi _{R_{k}\cap \Omega }\big\|_{q}^{r} \\
&=I_{1}+I_{2}.
\end{align*}

\textbf{Estimate of }$I_{1}$\textit{.} Let $\varrho $ be as in Lemma \ref%
{Sobolev-rep1}. We decompose $I_{1}$\ as follows: $I_{1}=I_{3}+I_{4}$, where%
\begin{equation*}
I_{3}=\sum\limits_{k\leq -1,\varrho \leq 2^{k-2}}2^{k\alpha _{1}r}\big\|%
f\chi _{R_{k}\cap \Omega }\big\|_{q}^{r}\quad \text{and}\quad
I_{4}=\sum\limits_{k\leq -1,\varrho >2^{k-2}}2^{k\alpha _{1}r}\big\|f\chi
_{R_{k}\cap \Omega }\big\|_{q}^{r}.
\end{equation*}%
Let $x\in R_{k}\cap \Omega $, $k\in \mathbb{Z}$. We estimate $I_{3}$. Since $%
x\in R_{k}\cap \Omega $ and $\varrho \leq 2^{k-2}$, we get $C_{x,\varrho
}\subset \tilde{R}_{k}=\{z:2^{k-2}\leq |z|\leq 2^{k+1}\}$. From Lemma \ref%
{Sobolev-rep1}, we easily obtain 
\begin{equation*}
|f(x)|\lesssim \sum_{|\beta |\leq m-1}\varrho ^{|\beta
|-n}\int_{C_{x,\varrho }}|D^{\beta }f(y)|\chi _{\tilde{R}_{k}}(y)dy+\sum_{|%
\beta |=m}\int_{C_{x,\varrho }}\frac{|D^{\beta }f(y)|}{|x-y|^{n-m}}\chi _{%
\tilde{R}_{k}}(y)dy,
\end{equation*}%
which is bounded by, because of\ $\alpha _{2}>\alpha _{1}$ and $m<n,$ 
\begin{align*}
&c2^{k(\alpha _{2}-\alpha _{1})}\sum_{|\beta |\leq m-1}\varrho ^{|\beta
|-n}\int_{C_{x,\varrho }}\frac{|D^{\beta }f(y)|}{|x-y|^{n-m+\alpha
_{2}-\alpha _{1}}}\chi _{\tilde{R}_{k}}(y)dy \\
&+c2^{k(\alpha _{2}-\alpha _{1})}\sum_{|\beta |=m}\int_{C_{x,\varrho }}\frac{%
|D^{\beta }f(y)|}{|x-y|^{n-m+\alpha _{2}-\alpha _{1}}}\chi _{\tilde{R}%
_{k}}(y)dy \\
&\lesssim 2^{k(\alpha _{2}-\alpha _{1})}\sum_{|\beta |\leq m-1}\mathcal{I}%
_{m-\alpha _{2}+\alpha _{1}}((D^{\beta }f)\chi _{\Omega \cap \tilde{R}%
_{k}})(x),
\end{align*}%
where the positive constant $c$ is independent of $k$. Thanks to Theorem \ref%
{result3} there exists some constant $c$ such that 
\begin{equation*}
I_{3}\lesssim \sum_{|\beta |\leq m}\big\|\mathcal{I}_{m-\alpha _{2}+\alpha
_{1}}((D^{\beta }f)\chi _{\Omega \cap \tilde{R}_{k}})\big\|_{\dot{K}%
_{q}^{\alpha _{2},r}(\mathbb{R}^{n})}^{r}\leq c\big\|f\big\|_{\dot{K}%
_{p,m}^{\alpha _{2},r}(\Omega )}^{r}.
\end{equation*}%
Now we estimate $I_{4}$. Let 
\begin{equation*}
J_{1,k}(x)=\sum_{|\beta |\leq m-1}\varrho ^{|\beta |-n}\int_{C_{x,\varrho
}}|D^{\beta }f(y)|dy\quad \text{and}\quad J_{2,k}(x)=\sum_{|\beta
|=m}\int_{C_{x,\varrho }}\frac{|D^{\beta }f(y)|}{|x-y|^{n-m}}dy.
\end{equation*}%
To estimate the first term we use the fact that $m<n$ and $\varrho >2^{k-2}$
which leads to%
\begin{align*}
J_{1,k}(x) &\lesssim 2^{(m-n)k}\sum_{|\beta |\leq
m-1}\int_{C_{x,2^{k-2}}}|D^{\beta }f(y)|dy \\
&+\sum_{|\beta |\leq m-1}\varrho ^{|\beta |-n}\int_{2^{k-2}\leq |x-y|\leq
\varrho }|D^{\beta }f(y)|\chi _{\Omega }(y)dy \\
&=c(J_{1,k}^{1}(x)+J_{1,k}^{2}(x)).
\end{align*}%
Let us estimate each term separately. By assumption %
\eqref{sobolevassumption1}\ and H\"{o}lder's inequality it is easy to see
that%
\begin{align*}
J_{1,k}^{1}(x) &\lesssim \sum_{|\beta |\leq
m-1}2^{(m-n)k}\int_{C_{x,2^{k-2}}}|D^{\beta }f(y)|\chi _{\tilde{R}%
_{k}}(y)\chi _{\Omega }(y)dy \\
&\lesssim 2^{k(\alpha _{2}-\alpha _{1}-\frac{n}{q})}\sum_{|\beta |\leq m-1}%
\big\|(D^{\beta }f)\chi _{\tilde{R}_{k}\cap \Omega }\big\|_{p}.
\end{align*}%
Therefore%
\begin{equation*}
\big\|(J_{1,k}^{1})\chi _{R_{k}\cap \Omega }\big\|_{q}\lesssim 2^{k(\alpha
_{2}-\alpha _{1})}\sum_{|\beta |\leq m-1}\big\|(D^{\beta }f)\chi _{\tilde{R}%
_{k}\cap \Omega }\big\|_{p}
\end{equation*}%
for any $k\leq -1$ such that $\varrho >2^{k-2}$. Rewriting $J_{1,k}^{2}$ as
follows: $J_{1,k}^{2}=J_{1,k,1}^{2}+J_{1,k,2}^{2}$, where%
\begin{equation*}
J_{1,k,1}^{2}(x)=\sum_{|\beta |\leq m-1}\varrho ^{|\beta
|-n}\int_{2^{k-2}\leq |x-y|\leq 2^{k+2}}|D^{\beta }f(y)|\chi _{\Omega }(y)dy
\end{equation*}%
and%
\begin{equation*}
J_{1,k,2}^{2}(x)=\sum_{|\beta |\leq m-1}\varrho ^{|\beta
|-n}\int_{2^{k+2}\leq |x-y|\leq \varrho }|D^{\beta }f(y)|\chi _{\Omega
}(y)dy.
\end{equation*}%
$J_{1,k,1}^{2}(x)$ can be estimated from above by 
\begin{equation*}
c2^{(\alpha _{2}-\alpha _{1})k}\sum_{|\beta |\leq m-1}\mathcal{I}_{m-\alpha
_{2}+\alpha _{1}}((D^{\beta }f)\chi _{\Omega })(x)
\end{equation*}%
for any $k\leq -1$ such that $\varrho >2^{k-2}$. Now we consider the second
term. We\ have%
\begin{equation*}
J_{1,k,2}^{2}(x)\lesssim \sum_{|\beta |\leq m-1}\int_{2^{k+2}\leq |x-y|\leq
\varrho }\frac{|D^{\beta }f(y)|\chi _{\Omega }(y)}{|x-y|^{n-m}}dy,
\end{equation*}%
which can be estimated by%
\begin{align*}
&c\sum_{|\beta |\leq m-1}\sum_{i=k+2}^{j}\int_{2^{i}\leq |x-y|\leq 2^{i+1}}%
\frac{|D^{\beta }f(y)|\chi _{\Omega }(y)}{|x-y|^{n-m}}\chi _{\tilde{R}%
_{i}}(y)dy \\
&\lesssim \sum_{|\beta |\leq m-1}\sum_{i=k+2}^{j}2^{(m-\frac{n}{p}-\alpha
_{2})i}2^{i\alpha _{2}}\big\|(D^{\beta }f)\chi _{\tilde{R}_{i}\cap \Omega }%
\big\|_{p},
\end{align*}%
where $2^{j-1}<\varrho \leq 2^{j},j\in \mathbb{Z}$ and we used H\"{o}lder's
inequality. By assumption \eqref{sobolevassumption1} we obtain%
\begin{equation*}
2^{k\alpha _{1}}\big\|\big(J_{1,k,2}^{2}\big)\chi _{R_{k}\cap \Omega }\big\|%
_{q}\lesssim 2^{(\frac{n}{p}-m+\alpha _{2})k}\sum_{|\beta |\leq
m-1}\sum_{i=k+2}^{j}2^{(m-\frac{n}{p}-\alpha _{2})i}2^{i\alpha _{2}}\big\|%
(D^{\beta }f)\chi _{\tilde{R}_{i}\cap \Omega }\big\|_{p}
\end{equation*}%
for any $k\leq -1$ such that $\varrho >2^{k-2}$.

We estimate $J_{2,k}$. We write 
\begin{equation}
J_{2,k}(x)=\sum_{|\beta |=m}\int_{C_{x,2^{k-2}}}\frac{|D^{\beta }f(y)|}{%
|x-y|^{n-m}}dy+\sum_{|\alpha |=m}\int_{B_{x,\varrho ,k}}\frac{|D^{\beta
}f(y)|}{|x-y|^{n-m}}dy,  \label{second}
\end{equation}%
where $B_{x,\varrho ,k}=C_{x}\cap \{y:2^{k-2}\leq |x-y|<\varrho \}$. The
first term is bounded by%
\begin{equation*}
c2^{k(\alpha _{2}-\alpha _{1})}\sum_{|\beta |\leq m}\int_{C_{x,2^{k-2}}}%
\frac{|D^{\beta }f(y)|\chi _{\Omega }(y)}{|x-y|^{n-m-\alpha _{1}+\alpha _{2}}%
}dy\lesssim 2^{k(\alpha _{2}-\alpha _{1})}\sum_{|\beta |\leq m}\mathcal{I}%
_{m-\alpha _{2}+\alpha _{1}}((D^{\beta }f)\chi _{\Omega })(x).
\end{equation*}%
Rewriting the second term of \eqref{second} as follows:\ $%
J_{2,k,1}+J_{2,k,2} $; where%
\begin{equation*}
J_{2,k,1}(x)=\sum_{|\beta |=m}\int_{2^{k-2}\leq |x-y|\leq 2^{k+2}}\frac{%
|D^{\beta }f(y)|}{|x-y|^{n-m}}dy
\end{equation*}%
and%
\begin{equation*}
J_{2,k,2}(x)=\sum_{|\beta |=m}\int_{2^{k+2}\leq |x-y|\leq \varrho }\frac{%
|D^{\beta }f(y)|}{|x-y|^{n-m}}dy.
\end{equation*}%
Observe that%
\begin{equation*}
J_{2,k,1}(x)\lesssim 2^{k(\alpha _{2}-\alpha _{1})}\sum_{|\beta |=m}\mathcal{%
I}_{m-\alpha _{2}+\alpha _{1}}((D^{\beta }f)\chi _{\Omega })(x).
\end{equation*}%
As in the estimation of $J_{1,k,2}^{2}$, we obtain%
\begin{equation*}
2^{k\alpha _{1}}\big\|\big(J_{2,k,2}\big)\chi _{R_{k}\cap \Omega }\big\|%
_{q}\lesssim 2^{(\frac{n}{p}-m+\alpha _{2})k}\sum_{|\beta |\leq
m-1}\sum_{i=k+2}^{j}2^{(m-\frac{n}{p}-\alpha _{2})i}2^{i\alpha _{2}}\big\|%
(D^{\beta }f)\chi _{\tilde{R}_{i}\cap \Omega }\big\|_{p}.
\end{equation*}%
Using the fact that\ $\alpha _{2}>m-\frac{n}{p}$, we obtain by Lemma \ref%
{lq-inequality} that $I_{4}\leq c\big\|f\big\|_{\dot{K}_{p,m}^{\alpha
_{2},r}(\Omega )}^{r}$.

\textbf{Estimate of }$I_{2}$\textit{.} Since $\alpha _{2}>\alpha _{1}$, we
obtain that 
\begin{equation*}
I_{2}\leq \sup_{k\in \mathbb{N}_{0}}2^{kr\alpha _{2}}\big\|f\chi _{\Omega }%
\big\|_{q}^{r}\lesssim \big\|f\big\|_{\dot{K}_{q}^{\alpha _{2},\infty
}(\Omega )}^{r}.
\end{equation*}%
Again from Lemma \ref{Sobolev-rep1},%
\begin{equation*}
|f(x)|\lesssim \sum_{|\beta |\leq m}\mathcal{I}_{m-\alpha _{2}+\alpha
_{1}}((D^{\beta }f)\chi _{\Omega })(x),\quad x\in \Omega .
\end{equation*}%
Using again Theorem \ref{result3} it follows as above that%
\begin{align*}
I_{2} &\lesssim \sum_{|\beta |\leq m}\big\|\mathcal{I}_{m-\alpha _{2}+\alpha
_{1}}((D^{\beta }f)\chi _{\Omega })\big\|_{\dot{K}_{q}^{\alpha _{2},r}(%
\mathbb{R}^{n})}^{r} \\
&\lesssim \sum_{|\beta |\leq m}\big\|(D^{\beta }f)\chi _{\Omega }\big\|_{%
\dot{K}_{p}^{\alpha _{2},r}(\mathbb{R}^{n})}^{r} \\
&\lesssim \big\|f\big\|_{\dot{K}_{p,m}^{\alpha _{2},r}(\Omega )}^{r},
\end{align*}%
since $m-\frac{n}{p}<\alpha _{2}<n-\frac{n}{p}$. The proof is complete.
\end{proof}

\begin{remark}
We mention that Theorem \ref{embeddingsfirst} covers the Sobolev inequality.
In addition 
\begin{equation*}
W_{p}^{m}(\Omega ,|\cdot |^{\alpha _{2}p})\hookrightarrow \dot{K}%
_{q}^{\alpha _{1},p}(\Omega )\hookrightarrow L^{q}(\Omega ,|\cdot |^{\alpha
_{1}q}),
\end{equation*}%
under the same assumptions of Theorem \ref{embeddingsfirst} with $r=p$. In
particular%
\begin{equation*}
W_{p}^{m}(\Omega )\hookrightarrow \dot{K}_{q}^{0,p}(\Omega )\hookrightarrow
L^{q}(\Omega ),
\end{equation*}%
holds if $1<p<\infty ,0<m<\frac{n}{p}\ $and%
\begin{equation*}
\frac{n}{q}=\frac{n}{p}-m.
\end{equation*}
\end{remark}

\begin{theorem}
\label{embedingsq=p}Let domain $\Omega \subset \mathbb{R}^{n}$ satisfy the
cone condition , $0\in \Omega $ and $m\in \mathbb{N}$. Let $1<p<\infty
,1\leq r<\infty ,\alpha _{2}\geq \alpha _{1}$, $\alpha _{1}+\frac{n}{p}>0$
and 
\begin{equation*}
\max \Big(\frac{n}{p}+\alpha _{2},\frac{n}{p}+\alpha _{2}-\alpha _{1}\Big)%
<m<n.
\end{equation*}%
Then%
\begin{equation*}
\dot{K}_{p,m}^{\alpha _{2},r}(\Omega )\hookrightarrow \dot{K}_{p}^{\alpha
_{1},r}(\Omega )
\end{equation*}%
holds.
\end{theorem}

\begin{proof}
We use Theorem \ref{density}. Let $f\in C^{\infty }(\Omega )\cap \dot{K}%
_{p,m}^{\alpha _{2},r}(\Omega )$. We write%
\begin{align*}
\big\|f\big\|_{\dot{K}_{p}^{\alpha _{1},r}(\Omega )}^{r}&
=\sum\limits_{k=-\infty }^{\infty }2^{k\alpha _{1}r}\big\|f\chi _{R_{k}\cap
\Omega }\big\|_{p}^{r} \\
& =\sum\limits_{2^{k+2}>\varrho }2^{k\alpha _{1}r}\big\|f\chi _{R_{k}\cap
\Omega }\big\|_{p}^{r}+\sum\limits_{2^{k+2}\leq \varrho }2^{k\alpha _{1}r}%
\big\|f\chi _{R_{k}\cap \Omega }\big\|_{p}^{r} \\
& =I_{1}+I_{2}.
\end{align*}%
Let us estimate $I_{1}$. Let $t>0$ be such that $\frac{n}{m}<t<\min (p,\frac{%
n}{\max (0,\alpha _{2}+\frac{n}{p})})$. By H\"{o}lder's inequality, we
obtain 
\begin{align*}
|f(x)|& \lesssim \sum_{|\beta |\leq m-1}\varrho ^{|\beta
|-n}\int_{C_{x,\varrho }}|D^{\beta }f(y)|dy+\sum_{|\beta
|=m}\int_{C_{x,\varrho }}\frac{|D^{\beta }f(y)|}{|x-y|^{n-m}}dy \\
& \lesssim \sum_{|\beta |\leq m}\mathcal{M}_{t}((D^{\beta }f)\chi _{\Omega
})(x)
\end{align*}%
for any $x\in R_{k}\cap \Omega $. Therefore%
\begin{align*}
I_{1}& \lesssim \sum_{|\beta |\leq m}\sum\limits_{2^{k+2}>\varrho
}2^{k\alpha _{1}r}\big\|\mathcal{M}_{t}((D^{\beta }f)\chi _{\Omega })\chi
_{R_{k}}\big\|_{p}^{r} \\
& \lesssim \sum_{|\beta |\leq m}\big\|\mathcal{M}_{t}((D^{\beta }f)\chi
_{\Omega })\big\|_{\dot{K}_{p}^{\alpha _{2},r}(\mathbb{R}^{n})}^{r} \\
& \lesssim \big\|f\big\|_{\dot{K}_{p,m}^{\alpha _{2},r}(\Omega )}^{r},
\end{align*}%
by\ Lemma \ref{Maximal-Inq}.

Now we estimate $I_{2}$. We employ the same notation as in Theorem \ref%
{embeddingsfirst}. We have 
\begin{equation*}
J_{1,k}(x)\lesssim \sum_{|\beta |\leq m-1}\varrho ^{|\beta
|-n}\int_{C_{x,\varrho }}\frac{|D^{\beta }f(y)|}{|x-y|^{n-m}}dy,\quad x\in
R_{k}\cap \Omega .
\end{equation*}%
Therefore we need only to estimate $J_{2,k}$. We write 
\begin{equation}
J_{2,k}(x)=\sum_{|\beta |=m}\int_{C_{x,2^{k-2}}}\frac{|D^{\beta }f(y)|}{%
|x-y|^{n-m}}dy+\sum_{|\beta |=m}\int_{B_{x,\varrho ,k}}\frac{|D^{\beta }f(y)|%
}{|x-y|^{n-m}}dy,  \label{integral1}
\end{equation}%
where $B_{x,\varrho ,k}=C_{x}\cap \{y:2^{k-2}\leq |x-y|<\varrho \}$. Let $%
t>0 $ be such that $m-\frac{n}{t}+\alpha _{1}-\alpha _{2}>0$ and $t<p$. By H%
\"{o}lder's inequality the first integral of \eqref{integral1} is bounded by,%
\begin{align*}
&c2^{k(\alpha _{2}-\alpha _{1})}\int_{C_{x,2^{k-2}}}\frac{|D^{\beta
}f(y)|\chi _{\tilde{R}_{k}\cap \Omega }(y)}{|x-y|^{n-m-\alpha _{1}+\alpha
_{2}}}dy \\
&\lesssim 2^{k(\alpha _{2}-\alpha _{1})}2^{(m-\frac{n}{t}+\alpha _{1}-\alpha
_{2})k}\mathcal{M}_{t}((D^{\beta }f)\chi _{\tilde{R}_{k}\cap \Omega
})(x),\quad |\beta |=m.
\end{align*}%
The boundedness of the maximal function on $L^{\frac{p}{t}}(\mathbb{R}^{n})$
yield that%
\begin{equation*}
\big\|\mathcal{M}_{t}((D^{\beta }f)\chi _{\tilde{R}_{k}\cap \Omega })\big\|%
_{p}\lesssim \big\|(D^{\beta }f)\chi _{\tilde{R}_{k}\cap \Omega }\big\|%
_{p},\quad |\beta |=m.
\end{equation*}%
Now%
\begin{align*}
\int_{B_{x,2^{k+2},k}}\frac{|D^{\beta }f(y)|}{|x-y|^{n-m}}dy &\lesssim
2^{(m-n)k}\int_{B_{x,2^{k+2},k}}|D^{\beta }f(y)|\chi _{\Omega }(y)dy \\
&\lesssim 2^{mk}\mathcal{M}((D^{\beta }f)\chi _{\Omega })(x) \\
&\lesssim 2^{(m+\alpha _{1}-\alpha _{2})k}2^{(\alpha _{2}-\alpha _{1})k}%
\mathcal{M}((D^{\beta }f)\chi _{\Omega })(x),\quad |\beta |=m.
\end{align*}%
Let $j\in \mathbb{Z}$ be such that $2^{j-1}<\varrho \leq 2^{j}$. As in
Theorem \ref{embeddingsfirst} we obtain%
\begin{align*}
\int_{B_{x,\varrho ,k+4}}\frac{|D^{\beta }f(y)|}{|x-y|^{n-m}}dy &\lesssim
\sum_{i=k+2}^{j}2^{(m-\frac{n}{p}-\alpha _{2})i}2^{\alpha _{2}i}\big\|%
(D^{\beta }f)\chi _{R_{i}\cap \Omega }\big\|_{p} \\
&\lesssim \big\|f\big\|_{\dot{K}_{p;m}^{\alpha _{2},r}(\Omega )},\quad
|\beta |=m.
\end{align*}%
The desired estimate follows by Lemma \ref{Maximal-Inq} and the fact that $%
\alpha _{1}+\frac{n}{p}>0$. The proof is complete.
\end{proof}

\begin{theorem}
\label{embeddingsq=infinity}Let domain $\Omega \subset \mathbb{R}^{n}$
satisfy the cone condition, $0\in \Omega $ and $m\in \mathbb{N}_{0}$. Let $%
1<p<\infty ,1\leq r<\infty $ and $\frac{n}{p}+\alpha _{2}<m<n$. Assume that $%
\alpha _{2}\geq \alpha _{1}>0$. Then%
\begin{equation*}
\dot{K}_{p,m}^{\alpha _{2},r}(\Omega )\hookrightarrow \dot{K}_{\infty
}^{\alpha _{1},r}(\Omega )
\end{equation*}%
holds.
\end{theorem}

\begin{proof}
Let $f\in C^{\infty }(\Omega )\cap \dot{K}_{p,m}^{\alpha _{2},r}(\Omega )$.
We write%
\begin{align*}
\big\|f\big\|_{\dot{K}_{\infty }^{\alpha _{1},r}(\Omega )}^{r}&
=\sum\limits_{k=-\infty }^{\infty }2^{k\alpha _{1}r}\big\|f\chi _{R_{k}\cap
\Omega }\big\|_{\infty }^{r} \\
& =\sum\limits_{2^{k-2}>\varrho }2^{k\alpha _{1}r}\big\|f\chi _{R_{k}\cap
\Omega }\big\|_{\infty }^{r}+\sum\limits_{2^{k-2}\leq \varrho }2^{k\alpha
_{1}r}\big\|f\chi _{R_{k}\cap \Omega }\big\|_{\infty }^{r} \\
& =S_{1}+S_{2}.
\end{align*}

\textbf{Estimate of }$S_{1}$\textit{. }From Lemma \ref{Sobolev-rep1} and H%
\"{o}lder's inequality, because of $m>\frac{n}{p}$, we obtain%
\begin{equation*}
|f(x)|\lesssim \sum_{|\beta |\leq m}\big\|(D^{\beta }f)\chi _{\tilde{R}%
_{k}\cap \Omega }\big\|_{p}
\end{equation*}%
for any $x\in R_{k}\cap \Omega $, since $C_{x,\varrho }\subset \tilde{R}_{k}$%
. Hence%
\begin{align*}
S_{1} &\lesssim \sum_{|\beta |\leq m}\sum\limits_{2^{k-2}>\varrho
}2^{k(\alpha _{1}-\alpha _{2})r}2^{k\alpha _{2}r}\big\|(D^{\beta }f)\chi _{%
\tilde{R}_{k}\cap \Omega }\big\|_{p}^{r} \\
&\lesssim \big\|f\big\|_{\dot{K}_{p,m}^{\alpha _{2},r}(\Omega )}^{r},
\end{align*}%
because of $\alpha _{2}\geq \alpha _{1}$.

\textbf{Estimate of }$S_{2}$\textit{. }We have%
\begin{align*}
|f(x)| &\lesssim \sum_{|\beta |\leq m}\int_{C_{x,\varrho }}\frac{|D^{\beta
}f(y)|}{|x-y|^{n-m}}dy \\
&=\sum_{|\beta |\leq m}\int_{C_{x,2^{k-2}}}\frac{|D^{\beta }f(y)|}{%
|x-y|^{n-m}}dy+\sum_{|\beta |\leq m}\int_{B_{x,\varrho ,k}}\frac{|D^{\beta
}f(y)|}{|x-y|^{n-m}}dy \\
&=P_{1,k}(x)+P_{2,k}(x),
\end{align*}%
where $B_{x,\varrho ,k}=C_{x}\cap \{y:2^{k-2}\leq |x-y|<\varrho \}$. Using
again H\"{o}lder's inequality we obtain%
\begin{equation*}
P_{1,k}(x)\leq \sum_{|\beta |\leq m}\int_{C_{x,2^{k-2}}}\frac{|D^{\beta
}f(y)|}{|x-y|^{n-m}}\chi _{\tilde{R}_{k}\cap \Omega }(y)dy\lesssim 2^{k(m-%
\frac{n}{p})}\sum_{|\beta |\leq m}\big\|(D^{\beta }f)\chi _{\tilde{R}%
_{k}\cap \Omega }\big\|_{p},
\end{equation*}%
because of $m>\frac{n}{p}$. Therefore%
\begin{align*}
\sum\limits_{2^{k-2}\leq \varrho }2^{k\alpha _{1}r}\sup_{x\in R_{k}\cap
\Omega }(P_{1,k}(x))^{r} &\lesssim \sum_{|\beta |\leq
m}\sum\limits_{2^{k-2}\leq \varrho }2^{k(m-\frac{n}{p}+\alpha _{1}-\alpha
_{2})}2^{k\alpha _{2}r}\big\|(D^{\beta }f)\chi _{\tilde{R}_{k}\cap \Omega }%
\big\|_{p}^{r} \\
&\lesssim \big\|f\big\|_{\dot{K}_{p,m}^{\alpha _{2},r}(\Omega )}^{r},
\end{align*}%
since $m-\frac{n}{p}+\alpha _{1}-\alpha _{2}>0$. Now we estimate $P_{2,k}$.
We write $P_{2,k}=T_{1,k}+T_{2,k}+T_{3,k}$, where%
\begin{equation*}
T_{1,k}(x)=\sum_{|\beta |\leq m}\int_{B_{x,\varrho ,k}}\frac{|D^{\beta }f(y)|%
}{|x-y|^{n-m}}\chi _{|\cdot |\leq \frac{|x|}{2}}(y)dy,
\end{equation*}%
\begin{equation*}
T_{2,k}(x)=\sum_{|\beta |\leq m}\int_{B_{x,\varrho ,k}}\frac{|D^{\beta }f(y)|%
}{|x-y|^{n-m}}\chi _{\frac{|x|}{2}\leq |\cdot |\leq 2|x|}(y)dy
\end{equation*}%
and%
\begin{equation*}
T_{3,k}(x)=\sum_{|\beta |\leq m}\int_{B_{x,\varrho ,k}}\frac{|D^{\beta }f(y)|%
}{|x-y|^{n-m}}\chi _{|\cdot |>2|x|}(y)dy
\end{equation*}%
Let us consider the first term. Using the fact that $|x-y|\geq |y|$ if $%
|y|\leq \frac{|x|}{2}$ and H\"{o}lder's inequality to obtain%
\begin{align*}
T_{1,k}(x) &\lesssim \sum_{|\beta |\leq m}\int_{|y|\leq 2^{k}}\frac{%
|D^{\beta }f(y)|\chi _{\Omega }(y)}{|y|^{n-m}}dy \\
&=c\sum_{|\beta |\leq m}\sum_{i=-\infty }^{k}2^{i(m-\frac{n}{p.}-\alpha
_{2})}2^{i\alpha _{2}}\big\|(D^{\beta }f)\chi _{R_{i}\cap \Omega }\big\|_{p}
\\
&=c2^{k(m-\frac{n}{p.}-\alpha _{2})}\sum_{|\alpha |\leq m}\sum_{i=-\infty
}^{k}2^{(i-k)(m-\frac{n}{p.}-\alpha _{2})}2^{i\alpha _{2}}\big\|(D^{\beta
}f)\chi _{R_{i}\cap \Omega }\big\|_{p} \\
&\lesssim 2^{k(m-\frac{n}{p.}-\alpha _{2})}\big\|f\big\|_{\dot{K}%
_{p,m}^{\alpha _{2},r}(\Omega )},
\end{align*}%
since $m-\frac{n}{p.}-\alpha _{2}>0$. This leads to 
\begin{align*}
\sum\limits_{2^{k-2}\leq \varrho }2^{k\alpha _{1}r}\sup_{x\in R_{k}\cap
\Omega }(T_{1,k}(x))^{r} &\lesssim \big\|f\big\|_{\dot{K}_{p,m}^{\alpha
_{2},r}(\Omega )}^{r}\sum\limits_{2^{k-2}\leq \varrho }2^{k(m-\frac{n}{p.}%
-\alpha _{2}+\alpha _{1})r} \\
&\lesssim \big\|f\big\|_{\dot{K}_{p,m}^{\alpha _{2},r}(\Omega )}^{r}.
\end{align*}%
Now we easily obtain%
\begin{align*}
T_{2,k}(x) &\lesssim \sum_{|\alpha |\leq m}2^{k(m-n)}\int_{\Omega }|D^{\beta
}f(y)|\chi _{\tilde{R}_{k}}(y)dy \\
&\lesssim 2^{k(m-\frac{n}{p})}\sum_{|\alpha |\leq m}\big\|(D^{\beta }f)\chi
_{\tilde{R}_{k}\cap \Omega }\big\|_{p}
\end{align*}%
by H\"{o}lder's inequality. Therefore%
\begin{align*}
\sum\limits_{2^{k-2}\leq \varrho }2^{k\alpha _{1}r}\sup_{x\in R_{k}\cap
\Omega }(T_{2,k}(x))^{r} &\lesssim \sum_{|\beta |\leq
m}\sum\limits_{2^{k-2}\leq \varrho }2^{k(m-\frac{n}{p.}-\alpha _{2}+\alpha
_{1})}2^{k\alpha _{2}r}\big\|(D^{\beta }f)\chi _{\tilde{R}_{k}\cap \Omega }%
\big\|_{p}^{r} \\
&\lesssim \big\|f\big\|_{\dot{K}_{p,m}^{\alpha _{2},r}(\Omega )}^{r}.
\end{align*}%
Let us estimate $T_{3,k}$. We have $|x-y|\geq \frac{|y|}{2}$, if $|y|>2|x|$.
Then%
\begin{align*}
T_{3,k}(x) &\lesssim \sum_{|\beta |\leq m}\int_{2^{k}\leq |y|\leq 2\varrho }%
\frac{|D^{\beta }f(y)|}{|y|^{n-m}}\chi _{\Omega }(y)dy \\
&\lesssim \sum_{|\beta |\leq m}\sum_{i=k}^{j+1}2^{(m-\frac{n}{p}-\alpha
_{2})i}2^{i\alpha _{2}}\big\|(D^{\beta }f)\chi _{R_{i}\cap \Omega }\big\|_{p}
\\
&\lesssim \big\|f\big\|_{\dot{K}_{p,m}^{\alpha _{2},r}(\Omega )}^{r},
\end{align*}%
where $2^{j-1}<\varrho \leq 2^{j},j\in \mathbb{Z}$. Using the fact that $%
\alpha _{1}>0$ we obtain%
\begin{equation*}
\sum\limits_{2^{k-2}\leq \varrho }2^{k\alpha _{1}r}\sup_{x\in R_{k}\cap
\Omega }(T_{3,k}(x))^{r}\lesssim \big\|f\big\|_{\dot{K}_{p,m}^{\alpha
_{2},r}(\Omega )}^{r}.
\end{equation*}%
The proof is complete.
\end{proof}

Collecting the results obtained in Theorems \ref{embedingsq=p} and \ref%
{embeddingsq=infinity} we have the following statement.

\begin{theorem}
\label{embedingsp<q}Let domain $\Omega \subset \mathbb{R}^{n}$ satisfy the
cone condition, $0\in \Omega $ and $m\in \mathbb{N}_{0}$. Let $1<p<q<\infty
,1\leq r<\infty ,\alpha _{2}\geq \alpha _{1}>0$ and 
\begin{equation*}
\max \Big(\frac{n}{p}+\alpha _{2},\frac{n}{p}+\alpha _{2}-\alpha _{1}\Big)%
<m<n.
\end{equation*}%
Then%
\begin{equation*}
\dot{K}_{p,m}^{\alpha _{2},r}(\Omega )\hookrightarrow \dot{K}_{q}^{\alpha
_{1},r}(\Omega )
\end{equation*}%
holds.{}
\end{theorem}

\begin{proof}
Let $f\in \dot{K}_{p,m}^{\alpha _{2},r}(\Omega )$ and $\theta =\frac{p}{q}$.
We have%
\begin{equation*}
\big\|f\big\|_{\dot{K}_{q}^{\alpha _{1},r}(\Omega )}\leq \big\|f\big\|_{\dot{%
K}_{p}^{\alpha _{1},r}(\Omega )}^{\theta }\big\|f\big\|_{\dot{K}_{\infty
}^{\alpha _{1},r}(\Omega )}^{1-\theta }\lesssim \big\|f\big\|_{\dot{K}%
_{p,m}^{\alpha _{2},r}(\Omega )},
\end{equation*}%
by Theorems \ref{embedingsq=p} and \ref{embeddingsq=infinity}. The proof is
complete.
\end{proof}

In the previous results we have not treated the case $q<p$. The next theorem
gives a positive answer.

\begin{theorem}
\label{embedingsp<q copy(1)}Let domain $\Omega \subset \mathbb{R}^{n}$
satisfy the cone condition, $0\in \Omega $ and $m\in \mathbb{N}_{0}$. Let $%
1<q<p<\infty ,1\leq r<\infty ,\alpha _{2}+\frac{n}{p}\geq \alpha _{1}+\frac{n%
}{q}>0$ and 
\begin{equation*}
\max \Big(\frac{n}{p},\frac{n}{p}+\alpha _{2},\frac{n}{p}-\frac{n}{q}+\alpha
_{2}-\alpha _{1}\Big)<m<n.
\end{equation*}%
Then%
\begin{equation*}
\dot{K}_{p,m}^{\alpha _{2},r}(\Omega )\hookrightarrow \dot{K}_{q}^{\alpha
_{1},r}(\Omega )
\end{equation*}%
holds.{}
\end{theorem}

\begin{proof}
We use Theorem \ref{density}. Let $f\in C^{\infty }(\Omega )\cap \dot{K}%
_{p,m}^{\alpha _{2},r}(\Omega )$. We employ the same notation as in Theorem %
\ref{embedingsq=p}. Let us estimate $I_{1}$. Let $t>0$ be such that $1<\frac{%
n}{m}<t<\min (p,\frac{n}{\max (0,\alpha _{2}+\frac{n}{p})})$. We have 
\begin{align*}
|f(x)|& \lesssim \sum_{|\beta |\leq m-1}\varrho ^{|\beta
|-n}\int_{C_{x,\varrho }}|D^{\beta }f(y)|dy+\sum_{|\beta
|=m}\int_{C_{x,\varrho }}\frac{|D^{\beta }f(y)|}{|x-y|^{n-m}}dy \\
& \lesssim \sum_{|\beta |\leq m}\mathcal{M}_{t}((D^{\beta }f)\chi _{\tilde{R}%
_{k}\cap \Omega })(x)
\end{align*}%
for any $x\in R_{k}\cap \Omega $. H\"{o}lder's inequality together with the
boundedness of the maximal function on $L^{\frac{p}{t}}(\mathbb{R}^{n})$
leads to 
\begin{align*}
I_{1}& \lesssim \sum_{|\beta |\leq m}\sum\limits_{2^{k+2}>\varrho
}2^{(\alpha _{1}+\frac{n}{q}-\frac{n}{p}-\alpha _{2})kr}2^{k\alpha _{2}r}%
\big\|\mathcal{M}_{t}((D^{\beta }f)\chi _{\tilde{R}_{k}\cap \Omega })\big\|%
_{p}^{r} \\
& \lesssim \big\|f\big\|_{\dot{K}_{p,m}^{\alpha _{2},r}(\Omega )}^{r},
\end{align*}%
since $\alpha _{2}+\frac{n}{p}\geq \alpha _{1}+\frac{n}{q}.$

To estimate $I_{2}$ we need only to estimate $J_{2,k}$. Recall that 
\begin{equation*}
J_{2,k}(x)=\sum_{|\beta |=m}\int_{C_{x,2^{k-2}}}\frac{|D^{\beta }f(y)|}{%
|x-y|^{n-m}}dy+\sum_{|\beta |=m}\int_{B_{x,\varrho ,k}}\frac{|D^{\beta }f(y)|%
}{|x-y|^{n-m}}dy,
\end{equation*}%
where $B_{x,\varrho ,k}=C_{x}\cap \{y:2^{k-2}\leq |x-y|<\varrho \}$. By H%
\"{o}lder's inequality the first integral is bounded by%
\begin{equation*}
c\sum_{|\beta |=m}2^{k(m-\frac{n}{p})}\big\|(D^{\beta }f)\chi _{\tilde{R}%
_{k}\cap \Omega }\big\|_{p},
\end{equation*}%
where the positive constant $c$ is independent of $k$. Now%
\begin{align*}
\int_{B_{x,2^{k+2},k}}\frac{|D^{\beta }f(y)|}{|x-y|^{n-m}}dy &\lesssim
2^{(m-n)k}\int_{B_{x,2^{k+2},k}}|D^{\beta }f(y)|\chi _{\Omega }(y)dy \\
&\lesssim 2^{mk}\mathcal{M}((D^{\beta }f)\chi _{\Omega })(x),\quad |\beta
|=m.
\end{align*}%
Let $j\in \mathbb{Z}$ be such that $2^{j-1}<\varrho \leq 2^{j}$. As in
Theorem \ref{embeddingsfirst} we obtain%
\begin{align*}
\int_{B_{x,\varrho ,k+4}}\frac{|D^{\beta }f(y)|}{|x-y|^{n-m}}dy &\lesssim
\sum_{i=k+2}^{j}2^{(m-\frac{n}{p}-\alpha _{2})i}2^{\alpha _{2}i}\big\|%
(D^{\beta }f)\chi _{R_{i}\cap \Omega }\big\|_{p} \\
&\lesssim \big\|f\big\|_{\dot{K}_{p;m}^{\alpha _{2},r}(\Omega )},\quad
|\beta |=m.
\end{align*}%
Using H\"{o}lder's inequality and Lemma \ref{Maximal-Inq}, we get 
\begin{align*}
I_{2} &\lesssim \sum_{|\beta |\leq m}\sum\limits_{2^{k+2}\leq \varrho
}2^{k(m-\frac{n}{p}+\alpha _{1}-\alpha _{2}+\frac{n}{q})r}2^{k\alpha _{2}r}%
\Big(\big\|(D^{\beta }f)\chi _{\tilde{R}_{k}\cap \Omega }\big\|_{p}^{r}+%
\big\|\mathcal{M}((D^{\beta }f)\chi _{\Omega })\chi _{R_{k}}\big\|_{p}^{r}%
\Big) \\
&+\big\|f\big\|_{\dot{K}_{p;m}^{\alpha _{2},r}(\Omega
)}^{r}\sum\limits_{2^{k+2}\leq \varrho }2^{k(\alpha _{1}+\frac{n}{q})r} \\
&\lesssim \big\|f\big\|_{\dot{K}_{p;m}^{\alpha _{2},r}(\Omega )}^{r},
\end{align*}%
since $\alpha _{1}+\frac{n}{q}>0$ and $m-\frac{n}{p}+\alpha _{1}-\alpha _{2}+%
\frac{n}{q}>0$. The proof is complete.
\end{proof}

\textbf{Acknowledgements.} We would like to thank the referee for many
valuable comments and suggestions. This work is found by the General
Direction of Higher Education and Training under Grant No.
C00L03UN280120220004 and by The General Directorate of Scientific Research
and Technological Development, Algeria.


\bigskip

\bigskip

\bigskip

Douadi Drihem, M'sila University, Department of Mathematics,

Laboratory of Functional Analysis and Geometry of Spaces,

P.O. Box 166, M'sila 28000, Algeria,

e-mail: \texttt{\ douadidr@yahoo.fr, douadi.drihem@univ-msila.dz}

\end{document}